\documentclass{article} 
\usepackage{iclr2027_conference,times}

\usepackage{amsmath,amsfonts,bm}

\def\eqref#1{equation~\ref{#1}}

\def\1{\bm{1}}

\DeclareMathAlphabet{\mathsfit}{\encodingdefault}{\sfdefault}{m}{sl}
\SetMathAlphabet{\mathsfit}{bold}{\encodingdefault}{\sfdefault}{bx}{n}

\usepackage{hyperref}
\usepackage{natbib}
\usepackage{url}
\usepackage[utf8]{inputenc} 
\usepackage[T1]{fontenc}    
\usepackage{hyperref}       
\usepackage{url}            
\usepackage{booktabs}       
\usepackage{amsfonts}       
\usepackage{nicefrac}       
\usepackage{microtype}      
\usepackage[table]{xcolor}         

\usepackage{amsmath}
\usepackage{amssymb}
\usepackage{mathtools}
\usepackage{amsthm}
\usepackage[most]{tcolorbox}
\usepackage{mdframed}
\usepackage{thmtools}
\usepackage{wrapfig,subcaption}
\usepackage{xurl}
\usepackage{cancel}
\usepackage{algorithm}
\usepackage{algorithmic}

\usepackage{hyperref}
\usepackage[nameinlink,capitalize,noabbrev]{cleveref}
\definecolor{mypink}{rgb}{0.8353,0.251,0.3686}
\hypersetup{colorlinks=true,
    linkcolor=mypink,
    citecolor=mypink,
    filecolor=mypink,
    urlcolor=mypink
}

\theoremstyle{plain}
\newtheorem{theorem}{Theorem}[section]
\newtheorem{proposition}[theorem]{Proposition}
\newtheorem{lemma}[theorem]{Lemma}
\newtheorem{corollary}[theorem]{Corollary}
\newtheorem{example}[theorem]{Example}
\theoremstyle{definition}
\newtheoremstyle{myremark} 
  {3pt}   
  {3pt}   
  {\itshape} 
  {}      
  {\bf}      
  {.}     
  {.5em}  
  {}      

\theoremstyle{myremark}
\newtheorem{remark}{Remark}
\newtheorem{definition}[theorem]{Definition}

\newenvironment{explanation}[1][Explanation]{%
  \begin{trivlist}
  \item[\hskip\labelsep\itshape #1.]\normalfont
}{%
  \hfill\qedsymbol\end{trivlist}
}

\newtcolorbox{definitionbox}{
  enhanced,
  colback=white,
  colframe=black,
  boxrule=0pt,
  borderline west={1pt}{0pt}{black},
  left=8pt,
  right=6pt,
  top=6pt,
  bottom=6pt
}
\colorlet{ourblue}{blue!60}

\usepackage[textsize=tiny]{todonotes}

\title{Fast Differentiable SVD on GPU via \\ Polar Decomposition}

\iclrfinalcopy

\author{%
    Uliana Parkina\thanks{Equal contribution.} \\
    HSE University \\
    \texttt{uliana.parkina@gmail.com}
    \And
    Askar Tsyganov\footnotemark[1] \\
    HSE University \\
    \texttt{atsyganov@hse.ru}
    \And
    Sergei Kudriashov \\
    HSE University
    \And
    Sergey Samsonov \\
    HSE University
    \And
    Maxim Rakhuba$^\dagger$ \\
    HSE University
}

\begin{document}

\maketitle

\begin{abstract}
We present a fully GPU-oriented SVD pipeline based on polar decomposition, motivated by iterative methods that rely solely on matrix multiplications, such as the Newton-Schulz iteration. We show that this approach enables up to a $2\times$ speedup compared to standard implementations. Furthermore, we derive a numerically stable backward pass for the polar decomposition and leverage it to obtain a fully differentiable SVD. Our methods are released as open-source implementations in both PyTorch and JAX: \url{https://github.com/fallnlove/cans_svd}.
\end{abstract}

\section{Introduction}
\label{section:introduction}
Singular value decomposition (SVD) is a basic primitive in numerical linear algebra and a standard tool in machine learning. It underpins PCA~\cite{abdi2010principal,jolliffe2016principal}, global covariance pooling~\cite{wang2017global,wang2020deep}, tensor decompositions~\cite{oseledets2011tensor,vannieuwenhoven2012truncation}, and matrix-factorization-based recommender systems~\cite{cremonesi2010performance,yuan2019singular,rodpysh2023employing}.
In large scale deep learning, SVD is used for model compression and low-rank adaptation in large language models~\cite{meng2024pissa,wang2025svdllm,parkina2025coala}, optimization methods for neural networks~\cite{jordan2024muon,wang2018atomo}, and regularization techniques~\cite{shi2024domain}.

Standard algorithms for computing SVD, such as QR algorithm~\cite{demmel1990accurate} and Jacobi-type methods~\cite{hestenes1958inversion}, and their implementations in LAPACK~\cite{anderson1999lapack} and NVIDIA libraries~\cite{nvidia_cusolver_pdf} follow classical numerical linear algebra designs. These approaches rely on Householder reflections, Givens rotations, and QR factorizations, which scale poorly on modern GPUs. Many production implementations also target single precision arithmetic, which limits performance in lower precision regimes.

A factorization closely related to the SVD is the polar decomposition. In recent years, the polar decomposition has become increasingly popular in machine learning pipelines~\cite{jordan2024muon}, although it has not been fully explored in this context. We therefore turn our attention to this decomposition and develop a faster SVD algorithm via polar decomposition, using iterative methods for computing the polar factor, such as the Chebyshev-accelerated Newton--Schulz (CANS) iteration~\cite{grishina2025accelerating} and related approaches~\cite{amsel2025polar,chen2014stable}. These methods use only matrix multiplications and can run in lower precision, such as tensorfloat-32 (tf32) or bfloat-16 (bf16).

Another important challenge is that existing pipelines suffer from
numerical issues when directly implementing formulas for
differentiating the full SVD. For instance, while
PyTorch~\cite{paszke2019pytorch} supports differentiation through
the SVD, bounded gradient norms are guaranteed only for
differentiation through the singular values, which limits
the use of backpropagation through the full decomposition
despite its potential usefulness in deep
learning~\cite{qinsi2025dobisvd,wang2021robust}.
We address this issue by introducing a pseudo-gradient at
degenerate points and deriving regularized differentiation
formulas, showing that the polar decomposition is useful
in this context as well. To the best of our knowledge,
we are the first to derive a backward formula for the polar
decomposition and show that it does not rely on spectral gaps.
In contrast, standard SVD differentiation formulas become
ill-conditioned when singular values approach each other.
Our approach yields a regularized differentiation scheme
with controlled gradients for ill-conditioned or nearly
rank-deficient matrices. We also propose a regularization
to stabilize the backward pass of the eigenvalue decomposition
(EVD). Combining these components provides a backward pass
through the full SVD. Finally, experiments on a practical
LLM compression task demonstrate improved optimization
outcomes and model quality.

Our main contributions are as follows:
\begin{itemize}
    \item We introduce a new algorithm for computing the SVD via polar decomposition, achieving up to a $2\times$ speedup by using matrix multiplications in bf16/tf32 precision. We show that polar decomposition, combined with recent advances~\cite{grishina2025accelerating}, leads to a practical, end-to-end SVD method with meaningful performance gains on modern hardware.
    
    \item We develop a stable regularized approach for differentiating the polar decomposition and the SVD. We formally analyze gradient stability under perturbations and provide theoretical guarantees on how regularization improves the stability of the backward pass and controls the gradient norm, while converging to the unperturbed gradient.
    
\end{itemize}

The paper is organized as follows. Section~\ref{section:related_work} reviews related work, existing approaches to the problem, and their limitations. Section~\ref{section:preliminaries} introduces basic definitions, the connection between the polar decomposition and the SVD, and briefly describes iterative methods underlying the proposed computational routines. Sections~\ref{section:forward} and~\ref{section:backward} present the SVD computation and differentiation methods, respectively, and report separate experiments for each component.

\section{Related Work}
\label{section:related_work}
Computing the SVD is a standard task in numerical linear algebra. GPU-based approaches typically include QR-based algorithms, Jacobi-type methods, and polar-decomposition-based algorithms.
QR-based methods~\cite{golub2013matrix} first reduce a matrix to bidiagonal form via Householder reflections or Givens rotations and then apply implicit QR iterations. They are robust, stable, and underlie LAPACK SVD routines~\cite{demmel1990accurate}, but bidiagonalization creates long dependency chains, limiting scalability on GPU~\cite{struski2024efficient}.
Jacobi-based methods~\cite{hestenes1958inversion,de1989one} iteratively apply plane rotations to column pairs to reduce off-diagonal entries of the Gram matrix. Since independent pairs can be processed in parallel, these methods can outperform QR-based algorithms on GPUs for small and medium matrices~\cite{nvidia_cusolver_pdf}.
Methods based on polar decomposition \cite{higham1994parallel,nakatsukasa2013stable} compute the SVD by first finding the unitary polar factor and then applying an EVD to the remaining factor. Existing approaches, however, typically compute the polar factor using QR-based iterations, which limits their GPU applicability~\cite{nakatsukasa2013stable}.

\begin{remark}
\label{remark:demmel}
A simpler approach to computing the SVD of a matrix $M \in \mathbb{R}^{m \times n}$ is to form the Gram matrix $M^\top M$ and compute its EVD, taking singular values as the square roots of the eigenvalues. However, this approach is numerically unstable: forming $M^\top M$ squares the condition number of $M$, so the relative error in the singular values scales as $\mathcal{O}(\sqrt{\varepsilon_{\text{mach}}})$. This makes Gram matrix-based methods unreliable for ill-conditioned or nearly rank-deficient matrices, especially in differentiable settings where gradient stability is essential. See~\cite[p. 241]{demmel1997applied} for more details and experimental validation in Appendix~\ref{appendix:additional_exps:gram}.
\end{remark}

Differentiating matrix decompositions appears to be a much less studied problem than computing them. Most of the issues comes from their non-uniqueness for repeated eigenvalues or low-rank cases.
This problem has been addressed in the literature with analytic formulas for the reverse mode differentiation of the SVD presented in~\cite{ionescu2015matrix,townsend2016differentiating,kanchi2025differentiable}. However, these expressions have well-known failure modes: gradients from these formulas explode for rank-deficient matrices or when singular values are close, which makes their application problematic in practical settings.
Several works have aimed to address these issues.
First, \cite{wang2021robust} proposed to address the small spectral gap issue via a Taylor approximation. This allows us to bound the gradient norm by choosing an appropriate number of terms and to prevent blow-up when singular values are close. However, that work has focused predominantly on symmetric matrices, and the proposed method can yield biased gradients due to the approximation error from the Taylor remainder.
Second, \cite{zhang2024differentiable} proposed to handle both rank-deficient and small gap cases using the Moore--Penrose pseudoinverse. For full-rank matrices with a large gap this method is equivalent to the analytic solution from~\cite{townsend2016differentiating}. However, we observe that their derivation contains an error (see Appendix~\ref{appendix:example:explanation}), and their final gradient formula is incorrect.

\section{Preliminaries}
\label{section:preliminaries}
Throughout this paper, we focus only on reverse-mode automatic differentiation~\cite{griewank2008evaluating,Baydin2018automatic}, and unless stated otherwise, the term \emph{differentiation} refers to computing the backward pass.

\subsection{Notation}
\label{preleminaries:notation}

For any vector $x \in \mathbb{R}^{n}$, $\|x\|_p = \bigl(\sum_{i=1}^n |x_i|^p\bigr)^{1/p} $ denotes the $\ell_p$ norm ($p \geq 1$).
For any matrix $M \in \mathbb{R}^{m \times n}$, $M_{\cdot,i}$ denotes the $i$-th column of $M$, $\sigma_i(M)$ denotes the $i$-th singular value of $M$, and $\operatorname{vec}(M)$ denotes the vectorization of $M$.
We denote by $M^{+}$ the Moore--Penrose pseudoinverse of $M$, defined via the SVD $M = U\Sigma V^\top$ as $M^{+} = V\Sigma^{+}U^\top,
\;\Sigma^{+} = \mathrm{diag}\bigl(\sigma_1^{-1},\dots,\sigma_r^{-1},0,\dots,0\bigr),$
where $r = \operatorname{rank}(M)$.
For any square invertible matrix $M$, let $\kappa_2(M) = \|M\|_2 \|M^{-1}\|_2$ be the condition number of $M$, where $\|\cdot\|_2$ is the spectral norm. 
For matrices $A,B \in \mathbb{R}^{m \times n}$, the Hadamard (elementwise) product is denoted by $A \odot B,
(A \odot B)_{ij} = A_{ij} B_{ij}$.
Given a scalar-valued loss function $\mathcal L$, we denote by $\overline X = \partial \mathcal L / \partial X$ the gradient of $\mathcal L$ with respect to a variable $X$.
Notice that $\overline X$ has the same shape as $X$ and arises in the differential expression with the increment $\mathrm dX$~\cite{bright2025matrix}:
$\mathrm d\mathcal L = \langle \overline X, \mathrm dX \rangle
= \operatorname{tr}(\overline X^\top \mathrm dX).$ For a square matrix $A \in \mathbb{R}^{n \times n}$, we denote its symmetric part by
$
A_{\operatorname{sym}} = \tfrac12\bigl(A + A^\top\bigr)$.
We denote by $A_{\text{off-diag}}$ the matrix $A$ with its diagonal set to $0$.

\subsection{Polar Decomposition}

\begin{definition}[Polar decomposition]
\label{def:polar_decomp}
    Any matrix $M \in \mathbb{R}^{m \times n}$ with $m \ge n$ admits a (right) polar
    decomposition
$M = W H$,
    where $W \in \mathbb{R}^{m \times n}$ has orthonormal columns, i.e.,
    $W^\top W = I_n$, and $H \in \mathbb{R}^{n \times n}$ is symmetric positive
    semidefinite.
\end{definition}
The polar decomposition is closely related to the SVD. Let $M = U \Sigma V^\top$ be the SVD of $M \in \mathbb{R}^{m \times n}, m \geq n$, where $U \in \mathbb{R}^{m \times n}$ and
$V \in \mathbb{R}^{n \times n}$ have orthonormal columns and
$\Sigma \in \mathbb{R}^{n \times n}$ is diagonal with singular values
$\sigma_1 \ge \cdots \ge \sigma_n \ge 0$.
Using the SVD, $M$ can be rewritten as
\begin{equation}
\label{preleminaries:eq:svd_to_polar}
\textstyle
M = \underbrace{U V^\top}_{W}
    \underbrace{V \Sigma V^\top}_{H},
\end{equation}
which is exactly the Definition~\ref{def:polar_decomp}.
The matrix $H$ is symmetric positive semidefinite with eigenvalues being equal to singular values of $M$.
Moreover, its
eigenvectors are the right singular vectors of $M$.

\section{SVD Computation}
\label{section:forward}
While the polar decomposition can be obtained directly from the SVD, our interest lies in the opposite direction: computing the SVD via the polar decomposition~\cite{higham1994parallel,nakatsukasa2013stable,higham2015faster}. 

\paragraph{Our algorithm.} We first compute the polar factor $W$
and then recover the remaining SVD components from the symmetric factor $H = W^\top M$ using an EVD:
$H = V\Sigma V^\top$.
The last step is obtained by setting $U = WV$, and finally
\[
M = (WV)\Sigma V^\top = U\Sigma V^\top.
\]
There exist several iterative methods for computing the polar factor (e.g., Newton-Schulz iteration~\cite{kovarik1970some,bjorck1971iterative}). Recent works have proposed modified polynomial iterations, such as CANS~\cite{grishina2025accelerating} and Polar-Express~\cite{amsel2025polar}, which improve the practical efficiency of such methods. We use the CANS iteration in our algorithm based on the results of an ablation study (see Appendix~\ref{appendix:additional_exps:polar_factor} for details). A detailed description of the algorithm is provided in Algorithm~\ref{alg:svd_forward}.

\begin{remark}
    Although Muon~\cite{jordan2024muon} is often described as using Newton--Schulz iterations, this is not entirely accurate. Muon uses a different polynomial iteration, which is typically run for a small fixed number of steps and does not converge precisely to the polar factor, since an accurate polar decomposition is not required for the optimizer. As our goal is to compute the polar factor accurately, we cannot directly use the Muon polynomial for orthogonalization.
\end{remark}

Numerical stability of such iterations has been discussed in~\cite{amsel2025polar,chen2014stable}. The stability of CANS follows a pattern similar to that of classical Newton-Schulz-type iterations~\cite{nakatsukasa2012backward}. In addition, the parameters of CANS control a trade-off between convergence speed and finite-precision error: faster convergence comes at the cost of lower accuracy. We analyze this trade-off for different arithmetic precisions and choose the parameters that give the best balance in our experiments (see Appendix~\ref{appendix:additional_exps:delta} for details).

\begin{figure}[ht]
    \centering
    \begin{subfigure}[t]{0.32\textwidth}
        \includegraphics[width=\linewidth]{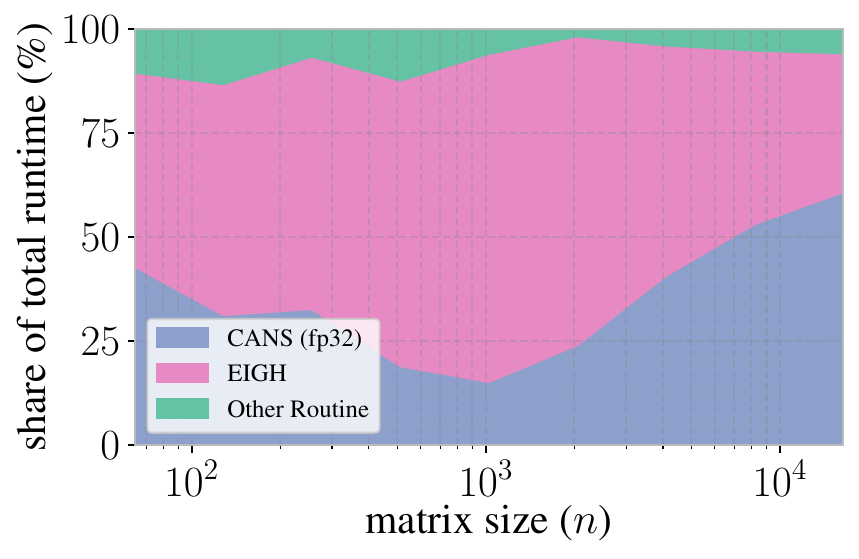}
        \caption{CANS SVD (fp32) algorithm}
        \label{fig:exp:forward:runtime:1}
    \end{subfigure}
    ~
    \begin{subfigure}[t]{0.32\textwidth}
        \includegraphics[width=\linewidth]{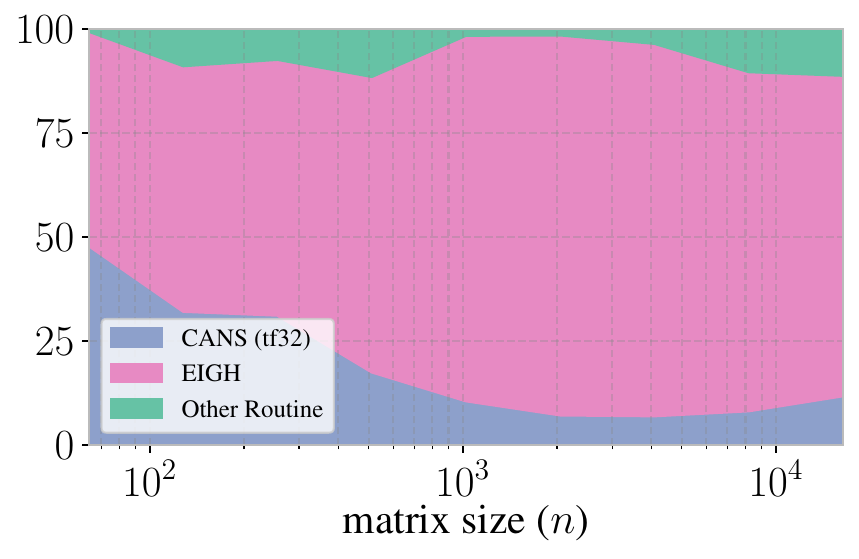}
        \caption{CANS SVD (tf32) algorithm}
        \label{fig:exp:forward:runtime:2}
    \end{subfigure}
    ~
    \begin{subfigure}[t]{0.32\textwidth}
        \includegraphics[width=\linewidth]{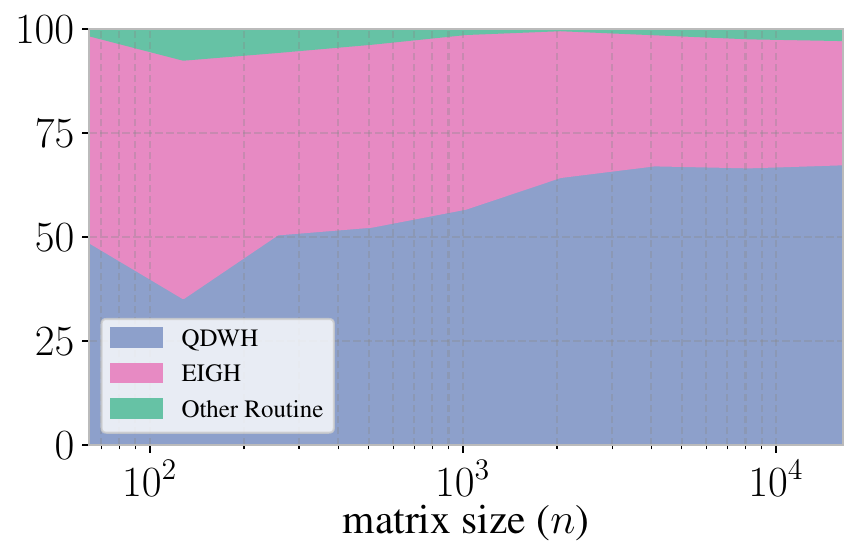}
        \caption{QDWH SVD algorithm}
        \label{fig:exp:forward:runtime:3}
    \end{subfigure}
    \caption{Comparison of different parts of the SVD computation by their share of the total runtime for a random matrix $M \in \mathbb{R}^{n \times n}$ with $\kappa_2(M) = 10$ on an NVIDIA B200 GPU. We report the shares of polar-factor computation (in blue), symmetric EVD (EIGH) (in pink), and other routines (in green) for different polar-decomposition-based SVD algorithms.}
    \label{fig:exp:forward:runtime}
\end{figure}

\begin{algorithm}[H]
\caption{CANS SVD. Steps marked in {\color{ourblue} blue color} can be executed in half precision}
\label{alg:svd_forward}
\begin{algorithmic}[1]
\REQUIRE
 $M \in \mathbb{R}^{m \times n}$ ($m \geq n$), tolerance $\varepsilon_{\texttt{QR}} > 0$
\ENSURE
$U, \Sigma, V^\top$ -- SVD of $M$
\STATE Normalize $M$ before NS iterations \COMMENT{$\rhd$ see Appendix~\ref{appendix:impl_details:ns_init}}
\STATE {\color{ourblue} Perform CANS preprocessing on $M$}
\STATE {\color{ourblue} Compute polar factor $W$ using CANS}
\STATE {\color{ourblue} $H \gets W^\top M$}
\STATE Compute EVD of $H = V \Sigma V^\top$
\STATE {\color{ourblue} $U \gets W V$}
\IF{$\exists$ $i$ such that $\bigl |\| U_{\cdot,i}\|_2 - 1 \bigr | > \varepsilon_{\texttt{QR}}$}
\STATE Orthogonalize $U$ using QR \COMMENT{$\rhd$ see Appendix~\ref{appendix:impl_details:rank_deficient}}
\ENDIF
\RETURN $U, \Sigma, V^\top$
\end{algorithmic}
\end{algorithm}

As an alternative to Newton--Schulz-type iterations, the QDWH iteration~\cite{nakatsukasa2010optimizing} has been widely used in practice in recent decades for computing the polar factor due to its cubic convergence rate. However, QDWH relies heavily on repeated QR factorizations, which scale poorly on GPUs (see Appendix~\ref{appendix:additional_exps:scaling}). As a result, in our experiments, the QDWH computation accounts for more than half of the total SVD runtime on modern GPUs (see Figure~\ref{fig:exp:forward:runtime:3}). Additionally, CANS uses only matrix multiplications, which allows us to exploit efficient bf16 and tf32 matrix multiplication routines, substantially accelerating the SVD and shifting the bottleneck to the EVD step (see Figure~\ref{fig:exp:forward:runtime:2}). This makes CANS more suitable than QDWH for GPU oriented pipelines.

\subsection{Experiments}

In this section, we present an empirical evaluation of the proposed algorithm for computing the SVD. We compare the following methods:
\begin{itemize}
    \item \textbf{CANS SVD (fp32 and tf32\footnote{We additionally experiment with the bf16 version, but it provides only negligible speedup over tf32, see Appendix~\ref{appendix:additional_exps:tf32_vs_bf16}.} versions).} Algorithm~\ref{alg:svd_forward} computed in fp32 precision and in mixed precision (fp32 and tf32, wherever possible).
    \item \textbf{QDWH SVD.} Our implementation of SVD via polar decomposition using the QDWH iteration proposed in~\cite{nakatsukasa2013stable}.
    \item \textbf{CUDA POLAR.} cuSOLVER implementation of SVD via polar decomposition using the QDWH iteration (\textit{driver=gesvdp}).
    \item \textbf{CUDA QR.} cuSOLVER implementation of QR-based SVD (\textit{driver=gesvd}).
    \item \textbf{CUDA JACOBI.} cuSOLVER implementation of Jacobi-based SVD (\textit{driver=gesvdj}).
\end{itemize}

Importantly, CUDA POLAR and CUDA QR do not provide any hyperparameters according to the cuSOLVER documentation\footnote{\url{https://docs.nvidia.com/cuda/cusolver/index.html}}. CUDA JACOBI provides hyperparameters controlling accuracy
when called from CUDA (see Appendix~\ref{appendix:additional_exps:jacobi} for details on choosing hyperparameters). Notably, JAX~\cite{jax2018github} and PyTorch~\cite{paszke2019pytorch} do not provide API access to set these parameters. The hyperparameters chosen for CANS SVD are provided in Appendix~\ref{appendix:impl_details:repro}.

\begin{figure}[ht]
    \begin{subfigure}[t]{0.43\textwidth}
    \centering
        \includegraphics[width=\linewidth]{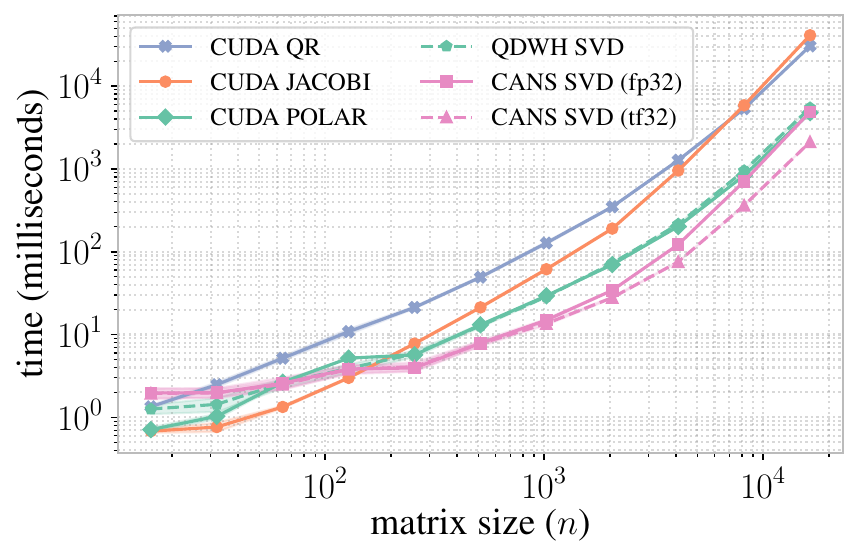}
        \caption{Square matrices $M$ with $\kappa_2(M) = 10$.}
        \label{fig:exp:forward:square}
    \end{subfigure}
    \hfill
    \begin{subfigure}[t]{0.43\textwidth}
    \centering
        \includegraphics[width=\linewidth]{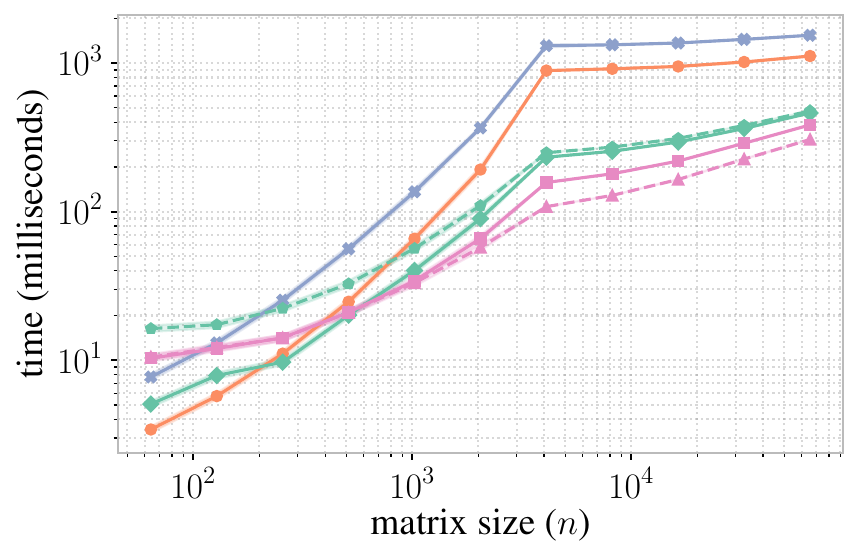}
        \caption{Rectangular matrices $M \in \mathbb{R}^{4096 \times n}$ with $\sigma_1(M) / \sigma_{\min(4096, n)}(M) = 10$.}
        \label{fig:exp:forward:rectangle}
    \end{subfigure}
    \vspace{-0.5em}
    \caption{Comparison of SVD algorithms on random matrices. The total wall-clock time is averaged over $100$ trials and reported with the standard deviation. \textbf{Note:} The algorithms have different accuracy levels, see Tables~\ref{tab:forward:recon_error} and~\ref{tab:forward:ortho_error} for a more comprehensive analysis. The experiments are conducted on an NVIDIA B200 GPU.}
    \label{fig:exp:forward}
\end{figure}

\paragraph{Square matrices.} We compare the accuracy of the algorithms for various condition numbers using fixed-size random square matrices in Tables~\ref{tab:forward:recon_error} and~\ref{tab:forward:ortho_error}. In Figure~\ref{fig:exp:forward:square}, we measure the performance of all algorithms on random square matrices $M \in \mathbb{R}^{n \times n}$ with $\kappa_2(M)=10$, as the size $n$ increases. The algorithms can be split into two groups by their accuracy: (i) less accurate algorithms (reconstruction error $\sim$ $10^{-3}$) that are fast - CANS SVD in mixed precision and CUDA JACOBI, and (ii) more accurate algorithms (reconstruction error $\sim$ $10^{-6}$) that are slow - CANS SVD in single precision, CUDA POLAR, and CUDA QR.

For the first group, CANS SVD (tf32 version) is preferable for medium-sized and large matrices, providing up to a $2 \times$ speedup compared to CUDA POLAR and up to a $10 \times$ speedup compared to CUDA JACOBI and CUDA QR. For small matrices, however, CUDA JACOBI is faster.

For the second group, CANS SVD in single precision is preferable for medium-sized and large matrices. For medium-sized matrices, it is up to $2 \times$ faster than the other baselines, while for $n = 2^{14}$, QDWH SVD and CUDA POLAR start to match the performance of CANS SVD. 
For small matrices, CUDA POLAR is faster than the other algorithms. However, CUDA POLAR failed to compute the SVD for matrices with condition numbers greater than $10^3$, while our re-implementation of this algorithm (QDWH SVD) performs well in this setting.

We note that our algorithms are implemented in JAX for benchmarking, so writing custom kernels may substantially speed them up and increase the gap with CUDA baselines for small matrices. Moreover, we expect that, with new generations of NVIDIA GPUs, CANS SVD can become even faster due to the trends observed in Appendix~\ref{appendix:additional_exps:scaling}.

\begin{table}[ht]
    \centering
    \small
    \caption{Comparison of the accuracy of SVD algorithms on a random square matrix $M \in \mathbb{R}^{4096 \times 4096}$ for different condition numbers. The reported metric is the relative reconstruction error $\|M - U\Sigma V^\top\|_F / \|M\|_F$, with standard deviation. Results are averaged over $100$ sampled matrices. \colorbox{red!10}{Red} indicates cases with significantly large errors.}
    \begin{tabular}{l|ccccc}
        \toprule
        Method & $10$ & $10^{2}$ & $10^{3}$ & $10^{4}$ & $10^{6}$ \\
        \midrule
        CANS SVD (fp32) & $4.9_{\color{gray}\scriptscriptstyle\pm .0} \cdot 10^{-6}$ & $4.2_{\color{gray}\scriptscriptstyle\pm .0} \cdot 10^{-6}$ & $4.3_{\color{gray}\scriptscriptstyle\pm .1} \cdot 10^{-6}$ & $4.4_{\color{gray}\scriptscriptstyle\pm .0} \cdot 10^{-6}$ & $5.3_{\color{gray}\scriptscriptstyle\pm .1} \cdot 10^{-6}$ \\
        CANS SVD (tf32) & $9.7_{\color{gray}\scriptscriptstyle\pm .0} \cdot 10^{-4}$ & $8.2_{\color{gray}\scriptscriptstyle\pm .0} \cdot 10^{-4}$ & $8.5_{\color{gray}\scriptscriptstyle\pm .0} \cdot 10^{-4}$ & $8.7_{\color{gray}\scriptscriptstyle\pm .0} \cdot 10^{-4}$ & $1.2_{\color{gray}\scriptscriptstyle\pm .0} \cdot 10^{-3}$ \\
        QDWH SVD & $3.4_{\color{gray}\scriptscriptstyle\pm .0} \cdot 10^{-6}$ & $3.3_{\color{gray}\scriptscriptstyle\pm .0} \cdot 10^{-6}$ & $3.7_{\color{gray}\scriptscriptstyle\pm .0} \cdot 10^{-6}$ & $5.7_{\color{gray}\scriptscriptstyle\pm .0} \cdot 10^{-6}$ & $6.2_{\color{gray}\scriptscriptstyle\pm .0} \cdot 10^{-6}$ \\
        CUDA POLAR & $4.3_{\color{gray}\scriptscriptstyle\pm .0} \cdot 10^{-6}$ & $5.5_{\color{gray}\scriptscriptstyle\pm .0} \cdot 10^{-6}$ & \colorbox{red!10}{$7.6_{\color{gray}\scriptscriptstyle\pm .0} \cdot 10^{1}$} & \colorbox{red!10}{$7.6_{\color{gray}\scriptscriptstyle\pm .0} \cdot 10^{1}$} & \colorbox{red!10}{$7.6_{\color{gray}\scriptscriptstyle\pm .0} \cdot 10^{1}$} \\
        CUDA QR & $1.8_{\color{gray}\scriptscriptstyle\pm .1} \cdot 10^{-5}$ & $2.7_{\color{gray}\scriptscriptstyle\pm .1} \cdot 10^{-5}$ & $1.5_{\color{gray}\scriptscriptstyle\pm .0} \cdot 10^{-5}$ & $1.6_{\color{gray}\scriptscriptstyle\pm .1} \cdot 10^{-5}$ & $1.1_{\color{gray}\scriptscriptstyle\pm .0} \cdot 10^{-5}$ \\
        CUDA JACOBI & $9.7_{\color{gray}\scriptscriptstyle\pm .4} \cdot 10^{-4}$ & $8.4_{\color{gray}\scriptscriptstyle\pm .4} \cdot 10^{-4}$ & $6.8_{\color{gray}\scriptscriptstyle\pm .8} \cdot 10^{-4}$ & $6.1_{\color{gray}\scriptscriptstyle\pm .6} \cdot 10^{-4}$ & $5.9_{\color{gray}\scriptscriptstyle\pm .7} \cdot 10^{-4}$ \\
        \bottomrule
    \end{tabular}
    \label{tab:forward:recon_error}
\end{table}

\begin{table}[ht]
\vspace{-0.5em}
    \centering
    \small
    \caption{Comparison of the accuracy of SVD algorithms on a random square matrix $M \in \mathbb{R}^{4096 \times 4096}$ for different condition numbers. The reported metric is the relative orthogonality error $\max\{\|U^\top U - I\|_F / \|I\|_F, \|V^\top V - I\|_F / \|I\|_F\}$, with standard deviation. Results are averaged over $100$ sampled matrices.}
    \begin{tabular}{l|ccccc}
        \toprule
        Method & $10$ & $10^{2}$ & $10^{3}$ & $10^{4}$ & $10^{6}$ \\
        \midrule
        CANS SVD (fp32) & $3.7_{\color{gray}\scriptscriptstyle\pm .0} \cdot 10^{-6}$ & $3.6_{\color{gray}\scriptscriptstyle\pm .0} \cdot 10^{-6}$ & $3.2_{\color{gray}\scriptscriptstyle\pm .9} \cdot 10^{-6}$ & $3.3_{\color{gray}\scriptscriptstyle\pm .0} \cdot 10^{-6}$ & $2.4_{\color{gray}\scriptscriptstyle\pm .0} \cdot 10^{-6}$ \\
        CANS SVD (tf32) & $6.6_{\color{gray}\scriptscriptstyle\pm .0} \cdot 10^{-4}$ & $6.2_{\color{gray}\scriptscriptstyle\pm .0} \cdot 10^{-4}$ & $2.8_{\color{gray}\scriptscriptstyle\pm .0} \cdot 10^{-6}$ & $6.2_{\color{gray}\scriptscriptstyle\pm .0} \cdot 10^{-4}$ & $6.2_{\color{gray}\scriptscriptstyle\pm .0} \cdot 10^{-4}$ \\
        QDWH SVD & $3.6_{\color{gray}\scriptscriptstyle\pm .0} \cdot 10^{-6}$ & $3.5_{\color{gray}\scriptscriptstyle\pm .0} \cdot 10^{-6}$ & $3.4_{\color{gray}\scriptscriptstyle\pm .0} \cdot 10^{-6}$ & $3.2_{\color{gray}\scriptscriptstyle\pm .0} \cdot 10^{-6}$ & $3.0_{\color{gray}\scriptscriptstyle\pm .0} \cdot 10^{-6}$ \\
        CUDA POLAR & $4.1_{\color{gray}\scriptscriptstyle\pm .0} \cdot 10^{-6}$ & $4.0_{\color{gray}\scriptscriptstyle\pm .0} \cdot 10^{-6}$ & $8.9_{\color{gray}\scriptscriptstyle\pm .0} \cdot 10^{-6}$ & $8.1_{\color{gray}\scriptscriptstyle\pm .0} \cdot 10^{-6}$ & $7.0_{\color{gray}\scriptscriptstyle\pm .0} \cdot 10^{-6}$ \\
        CUDA QR & $9.6_{\color{gray}\scriptscriptstyle\pm .2} \cdot 10^{-6}$ & $1.4_{\color{gray}\scriptscriptstyle\pm .0} \cdot 10^{-5}$ & $1.2_{\color{gray}\scriptscriptstyle\pm .0} \cdot 10^{-5}$ & $1.4_{\color{gray}\scriptscriptstyle\pm .1} \cdot 10^{-5}$ & $9.7_{\color{gray}\scriptscriptstyle\pm .4} \cdot 10^{-6}$ \\
        CUDA JACOBI & $2.1_{\color{gray}\scriptscriptstyle\pm .0} \cdot 10^{-3}$ & $2.1_{\color{gray}\scriptscriptstyle\pm .0} \cdot 10^{-3}$ & $2.2_{\color{gray}\scriptscriptstyle\pm .1} \cdot 10^{-3}$ & $2.3_{\color{gray}\scriptscriptstyle\pm .1} \cdot 10^{-3}$ & $2.6_{\color{gray}\scriptscriptstyle\pm .1} \cdot 10^{-3}$ \\
        \bottomrule
    \end{tabular}
    \label{tab:forward:ortho_error}
\end{table}

\paragraph{Rectangular matrices.} Lastly, we compare different algorithms for computing the SVD of rectangular matrices. We fix the first dimension of the matrix to $4096$ and vary the second dimension. For a more comprehensive analysis, we apply the QR preprocessing trick to all algorithms: we first compute the thin QR decomposition of the matrix and then compute the SVD of the triangular factor:
\begin{equation*}
M = Q {\color{ourblue} R } = Q {\color{ourblue} U \Sigma V^\top} = (Q U) \Sigma V^\top.
\end{equation*}Figure~\ref{fig:exp:forward:rectangle} presents the results. The differences from the square case are minor. For small matrices, CUDA JACOBI and CUDA POLAR provide the best performance-accuracy trade-offs with the optimal choice depending on the required accuracy. However, for medium and large matrices, CANS SVD consistently achieves higher performance.

\section{Stable Differentiation of Matrix Decompositions}
\label{section:backward}
In this work, we address two challenges in differentiating matrix
factorizations. First, at repeated eigenvalues or singular values,
the derivatives of the EVD and SVD factors are generally undefined,
and gradients may grow without bound as such points are approached.
We introduce a pseudo-gradient that specifies a backward rule
at repeated values and agrees with the classical backward
at nondegenerate points. Second, finite-precision arithmetic
introduces errors that can be strongly amplified by ill-conditioned
backward formulas, particularly near rank deficiency or small
spectral gaps. We regularize the potentially singular terms
in these formulas to control their magnitude and damp
the amplification of numerical errors.

\subsection{Differentiating the Polar Decomposition}

Although forward-mode differentiation (JVP) of the polar decomposition and generalized matrix functions has been studied in~\cite{gawlik2016computing,noferini2017generalizedfunction}, our focus lies in a different direction: reverse-mode differentiation (VJP) of decompositions. To the best of our knowledge, this is the first result to derive backward formulas for the polar decomposition. 

\begin{theorem}
\label{thm:polar_backward}
Let $M\in\mathbb{R}^{m\times n}$ ($m\ge n$) have full-rank.
Let $\mathcal L(M)=\ell(W,H)$ be a scalar loss and let $\overline{W}$ and $\overline H$ be the respective gradients.
Set $C:=(\overline H- W^\top \overline{W}\,H^{-1})_{\mathrm{sym}}$.
Define the mask matrix
\[
T_{ij}:=
\begin{cases}
\frac{1}{\sigma_i+\sigma_j}, & \sigma_i+\sigma_j>0,\\[2pt]
0, & \sigma_i+\sigma_j=0,
\end{cases}
\qquad 1\le i,j\le n,
\]
and $X:=V\Bigl(T\odot(V^\top C V)\Bigr)V^\top$.
Then 
$
\displaystyle
{
\ \overline M
\;=\;
\overline{W}\,H^{-1} \;+\; 2MX.
}
$
\end{theorem}
\paragraph{Discussion.} The proof is provided in Appendix~\ref{appendix:polar_formuls}. Differentiating the polar decomposition through the SVD introduces
inverse-gap terms whose analytical cancellation can be lost
in finite precision. This instability is discussed for the matrix
square root in~\cite{song2022fast}.
Our explicit backward formula avoids these terms: vanishing
singular value gaps introduce no additional ill-conditioning,
provided the smallest singular value stays bounded away from zero.

Theorem~\ref{thm:polar_backward} provides an exact backward formula for a
full-rank matrix $M$. However, this formula becomes increasingly ill-conditioned
as the smallest singular value of $M$ approaches zero. In finite-precision
arithmetic, such small singular values may be severely distorted by numerical
and accumulated errors even when both $M$ and its approximation remain full rank.
Thus the central question is:

\begin{center}
    How can we reliably evaluate the backward formula when $M$ is full rank but
    numerically close to rank deficiency?
\end{center}

In applications, we are interested in the gradient of a loss function with
respect to an exact matrix $M$, while in practice we only have access to a
perturbed matrix $\widetilde M$. Even when $\|M-\widetilde M\|$ is small, the
smallest singular values of the two matrices, and consequently the corresponding
gradients, may differ drastically.

We show that an appropriate regularization prevents the amplification of these
errors and provides a stable approximation of the desired gradient that converges
to the exact one as the perturbation vanishes.
\begin{proposition}
\label{prop:regularized_polar}
Let $M,\widetilde M\in\mathbb{R}^{m\times n}$, $m\geq n$, be full-rank matrices such that
$\|M-\widetilde M\|_F\leq\varepsilon$,
with polar decompositions $M=WH$ and $\widetilde M=\widetilde W\widetilde H$, where
$\widetilde H=\widetilde V\widetilde\Sigma\widetilde V^\top$,
$\widetilde\Sigma=\operatorname{diag}(\widetilde\sigma_1,\ldots,\widetilde\sigma_n)$.
For incoming gradients $\overline W,\overline H$, define
\[
\widetilde H_\varepsilon^{-1}
=
\widetilde V\widetilde\Sigma
(\widetilde\Sigma^2+\varepsilon^{1/2}I)^{-1}
\widetilde V^\top,
\qquad
\widetilde B_\varepsilon
=
\left(
\overline H-
\widetilde H_\varepsilon^{-1}
\widetilde M^\top\overline W
\widetilde H_\varepsilon^{-1}
\right)_{\mathrm{sym}},
\]
and
\[
\widetilde X_\varepsilon
=
\widetilde V
\left[
\frac{\widetilde\sigma_i+\widetilde\sigma_j}
     {(\widetilde\sigma_i+\widetilde\sigma_j)^2+\varepsilon^{1/4}}
(\widetilde V^\top\widetilde B_\varepsilon\widetilde V)_{ij}
\right]_{i,j=1}^n
\widetilde V^\top.
\]
Then the regularized gradient
$
\overline{\widetilde M}_\varepsilon
=
\overline W\widetilde H_\varepsilon^{-1}
+
2\widetilde M\widetilde X_\varepsilon$
satisfies, for all sufficiently small $\varepsilon>0$:
$
\|\overline M-\overline{\widetilde M}_\varepsilon\|_F
\leq c\,\varepsilon^{1/4}$,
where $c=c(M,\overline W,\overline H)$ is independent of $\varepsilon$.
\end{proposition}
The proof is provided in Appendix~\ref{appendix:regularized_polar}.
This result suggests that the regularized gradient
$\overline{\widetilde M}_\varepsilon$ at the perturbed point converges to the exact gradient for the unperturbed point $\overline M$
as $\varepsilon \to 0$.
The final algorithm is given in Algorithm~\ref{alg:polar_backward}.

\begin{algorithm}[H]
\caption{Polar Decomposition Backward}
\label{alg:polar_backward}
\begin{algorithmic}[1]
\REQUIRE
 $M \in \mathbb{R}^{m \times n}$ with polar decomposition $M = W H$, gradients
$\overline{W},\overline{H}$, and $\varepsilon > 0$
\ENSURE
Gradient $\overline{M}$
\vspace{0.4em}
\STATE EVD of $H = V \Sigma V^\top$
\vspace{0.2em}
\STATE $H_{\varepsilon}^{-1} \gets V \left( \Sigma (\Sigma^2 +\varepsilon^{1/2} I)^{-1} \right) V^\top$
\vspace{0.2em}
\STATE
$
B \gets (\overline{H} - H^{-1}_\varepsilon M^\top \overline{W} H^{-1}_{\varepsilon})_{\operatorname{sym}}
$ \label{alg:normalize}
\vspace{0.2em}
\STATE
$
\widetilde{X}_{ij} \gets \frac{(\sigma_i + \sigma_j)}{(\sigma_i + \sigma_j)^2 + \varepsilon^{1/4}} \cdot (V^\top B V)_{ij}
$
\\
\vspace{0.2em}
\STATE $X \gets V \widetilde{X} V^\top$
\vspace{0.2em}
\STATE
$
\overline{M} \gets \overline{W} H^{-1}_{\varepsilon} + 2 M X
$
\vspace{0.2em}
\RETURN $\overline{M}$
\end{algorithmic}
\end{algorithm}

\subsection{Differentiating the SVD}
\label{subsect:polar_svd}

Recall the connection between the polar decomposition and the SVD~\eqref{preleminaries:eq:svd_to_polar}. 
Therefore, to differentiate the SVD it only remains to differentiate the EVD of $H$.

Although differentiation through the EVD and SVD is classical
and well studied~\cite{giles2008extended}
(see also Appendix~\ref{appendix:eigenvalue_dev}),
the derivatives of the decomposition factors are generally
undefined at repeated eigenvalues or singular values.
Moreover, gradients can become arbitrarily large as spectral
gaps vanish, causing severe numerical instability~\cite{wang2021robust}.
In this work, we introduce a pseudo-gradient that provides
a backward rule at such degenerate points and agrees with
the classical backward wherever the decomposition factors
are differentiable.

\begin{definition}[EVD pseudo-gradient]
\label{def:minimum-norm-evd-gradient}
\itshape
Let $M \in \mathbb{R}^{n \times n}$ be symmetric positive definite,
and fix an eigendecomposition $M = V\Sigma V^\top$.
Let $\overline V$ and the diagonal matrix $\overline\Sigma$
denote the incoming gradients of a differentiable scalar loss
$\mathcal{L}(M) = \ell(V,\Sigma)$.

The EVD pseudo-gradient at this factorization,
denoted by $\overline M_\star$, is the solution of
\[
    \min_{G = G^\top} \|G\|_F^2
\qquad
\text{subject to}
\qquad
    \langle G,dM\rangle_F
    =
    \langle \overline V,\underbrace{V\Omega}_{dV}\rangle_F
    +
    \langle \overline\Sigma,\underbrace{D}_{d\Sigma}\rangle_F
\]
for every diagonal \(D\) and skew-symmetric \(\Omega\) with \(\Omega_{ij}=0\) if \(\sigma_i=\sigma_j\).
\end{definition}

The Definition~\ref{def:minimum-norm-evd-gradient} selects the symmetric matrix of minimum Frobenius
norm among those reproducing the loss differential along the
specified factor perturbations. The restriction on $\Omega$
excludes rotations within repeated eigenvalue eigenspaces:
these rotations commute with $\Sigma$ and therefore leave $M$
unchanged, although they may change $\ell(V,\Sigma)$.
Their contribution to the loss differential is disregarded
by definition, rather than assumed to vanish. Existence and uniqueness are established in Appendix~\ref{appendix:eigenvalue_dev}.

\begin{theorem}
\label{thm:eigh_backward_blocks}
Let $M\in\mathbb{R}^{n\times n}$ be symmetric positive definite,
and fix an EVD: $M=V\Sigma V^\top$,
    $\Sigma=\operatorname{diag}(\sigma_1,\ldots,\sigma_n)$.
Let $\overline V$ and the diagonal matrix $\overline\Sigma$
be the incoming gradients of a differentiable scalar loss
$\mathcal{L}(M) = \ell(V,\Sigma)$. Define
\[
    K_{ij}=
    \begin{cases}
        \dfrac{1}{\sigma_i-\sigma_j}, & \sigma_i\neq\sigma_j,\\[6pt]
        0, & \sigma_i=\sigma_j,
    \end{cases}
    \qquad
    S=\frac12 K^\top\odot
    \bigl(V^\top\overline V-\overline V^\top V\bigr).
\]
Then the EVD pseudo-gradient from
Definition~\ref{def:minimum-norm-evd-gradient} is given by
$\overline M_\star=V(S+\overline\Sigma)V^\top$.
\end{theorem}

Proof is provided in Appendix~\ref{appendix:eigenvalue_dev}. Similarly to the polar decomposition, in practical computations we operate in
finite-precision arithmetic and therefore cannot reliably determine whether
eigenvalues are exactly equal. Consequently, the notion of equality between
eigenvalues must be interpreted relative to machine precision.

By analogy with Proposition~\ref{prop:regularized_polar}, the same
principle can be applied to the differentiation of the EVD. In particular, the
only potentially singular object in the backward formula is the matrix
$K$, which encodes the solution of the commutator (Lyapunov-type)
equation
$
\Sigma S - S \Sigma = -B,
$ 
where $B$ is defined in Algorithm~\ref{alg:polar_backward}. As in the polar case, this equation is singular whenever eigenvalues coincide,
and its solution is not unique.

We introduce a smooth $\varepsilon$-regularization of $K$ by setting
$(K_\varepsilon)_{ij}
= \frac{\sigma_i-\sigma_j}{(\sigma_i-\sigma_j)^2+\varepsilon^2}$,
where $\varepsilon>0$.
This corresponds to Tikhonov regularization of the associated
linear system. For each $\varepsilon>0$, the resulting backward
formula depends smoothly on the factors and incoming gradients.
For a fixed eigendecomposition and fixed incoming gradients,
it converges to the pseudo-gradient from
Theorem~\ref{thm:eigh_backward_blocks} as $\varepsilon\to0$. The full pipeline is presented in Appendix~\ref{appendix:pipeline}.

\begin{table}[t]
    \centering
    \caption{Comparison of the gradient norm of SVD backward pass algorithms on a random square matrix $M \in \mathbb{R}^{4096 \times 4096}$ for different spectral gaps. The reported metric is the Frobenius norm of the gradient $\| \overline{M} \|_F$. Results are averaged over $100$ sampled matrices.}
    \resizebox{\textwidth}{!}{%
    \begin{tabular}{l|ccccccc}
        \toprule
        Method \textbackslash \ Spec.gap & $1$ & $10^{-1}$ & $10^{-2}$ & $10^{-3}$ & $10^{-4}$ & $10^{-5}$ & $10^{-6}$ \\
        \midrule
        Analytical~\cite{townsend2016differentiating} &
        $66.0_{\color{gray}\scriptscriptstyle\pm .6}$ &
        $67.0_{\color{gray}\scriptscriptstyle\pm 1.0}$ &
        $100.0_{\color{gray}\scriptscriptstyle\pm 41.0}$ &
        $690.0_{\color{gray}\scriptscriptstyle\pm 580.0}$ &
        $5500.0_{\color{gray}\scriptscriptstyle\pm 4600.0}$ &
        NaN &
        NaN \\
        Taylor~\cite{wang2021robust} &
        $49.0_{\color{gray}\scriptscriptstyle\pm .3}$ &
        $49.0_{\color{gray}\scriptscriptstyle\pm .3}$ &
        $49.0_{\color{gray}\scriptscriptstyle\pm .3}$ &
        $49.0_{\color{gray}\scriptscriptstyle\pm .3}$ &
        $49.0_{\color{gray}\scriptscriptstyle\pm .3}$ &
        $49.0_{\color{gray}\scriptscriptstyle\pm .3}$ &
        $49.0_{\color{gray}\scriptscriptstyle\pm .3}$ \\
        Tikhonov (ours) &
        $66.0_{\color{gray}\scriptscriptstyle\pm .6}$ &
        $67.0_{\color{gray}\scriptscriptstyle\pm 1.0}$ &
        $100.0_{\color{gray}\scriptscriptstyle\pm 41.0}$ &
        $690.0_{\color{gray}\scriptscriptstyle\pm 580.0}$ &
        $5500.0_{\color{gray}\scriptscriptstyle\pm 4600.0}$ &
        $66.0_{\color{gray}\scriptscriptstyle\pm .6}$ &
        $66.0_{\color{gray}\scriptscriptstyle\pm .6}$ \\
        \bottomrule
    \end{tabular}}
    \label{tab:backward:norm}
\end{table}

\subsection{Experiments}

In this section, we present an empirical evaluation of the proposed algorithm for computing the backward-pass SVD.

\paragraph{Synthetic Experiments.}  First, we validate the correctness of the proposed algorithm. We compare the following methods:
\begin{itemize}
    \item \textbf{Analytical formula~\cite{townsend2016differentiating}}. The standard closed-form SVD backward expression.
    \item \textbf{Taylor approximation~\cite{wang2021robust}}. A Taylor series approximation-based method.
    \item \textbf{Our method}. We implement differentiation of the SVD directly from our formulas for the polar decomposition and the EVD. The arising linear equations are solved using Tikhonov regularization\footnote{Although we do not provide a dedicated theoretical analysis, the Tikhonov regularization in the backward pass can be replaced by a $\delta$-regularized Moore--Penrose pseudoinverse~\cite{golub2013matrix}, yielding an alternative well-defined and numerically stable gradient in degenerate or ill-conditioned cases.}; see Algorithm~\ref{alg:polar_backward} (we set $\varepsilon$ to $10^{-8}$).
\end{itemize}

We compare the gradient norms of the methods for various spectral gaps. Gradients are computed for the following loss function:
$
    \mathcal{L}(U, \Sigma, V) = 
    \langle A, U \rangle +
    \langle B, \Sigma \rangle +
    \langle C, V \rangle$.
We sample $A, B, C$ with i.i.d. entries from the standard normal distribution.
Table~\ref{tab:backward:norm} presents the results of differentiating the loss function. Gradients in the analytical formula explode when the spectral gap is less than $10^{-4}$ and produce NaNs. Both of our methods, as well as the Taylor approximation, are stable when applied to matrices with close singular values. However, the Taylor approximation is inaccurate: its norm does not match the analytical formula~\cite{townsend2016differentiating} even for large spectral gaps. Both of our methods provide gradients with the same norm as the standard closed-form expression in the case of large spectral gaps. Additionally, we validate our implementation with finite-difference checks in the code.

\paragraph{LLM Compression.} Here, we demonstrate how our formulas can be applied to large language model compression. We follow the Dobi-SVD pipeline~\cite{qinsi2025dobisvd}. Instead of compressing all weight matrices with the same ratio, as in classical approaches~\cite{parkina2025coala}, we learn layer-specific ranks by differentiating through the SVD of activations. More precisely, let $\{W_i, X_i\}_{i = 1}^n$ denote the weight matrices of linear layers and its inputs respectively. Each layer has a trainable parameter $\gamma_i \in \mathbb{R}_+$, interpreted as a continuous relaxation of its rank. At each training step, for an input batch $X_i$, we truncate the activations $W_i X_i$ using the SVD and minimize the loss with a regularizer that promotes the desired compression ratio.

We use the LLaMA-2-7B model~\cite{touvron2023llama} with 256 random training samples and 256 validation samples from the WikiText-2 dataset~\cite{merity2016pointer}. For a clearer comparison of SVD pipelines, we do not perform quantization or weight-update steps, unlike Dobi-SVD~\cite{qinsi2025dobisvd}. We replace the Taylor-based differentiation used in Dobi-SVD with our backward pass. To ablate the effect of regularization, we test several values of $\varepsilon$.

\begin{table}
\small
  \caption{Metric values of various compression methods. Experiments were conducted using the \emph{LLaMA-2-7B} model on the WikiText-2 dataset.}
    \centering
    \begin{tabular}{c l c c c}
      \toprule
      \textbf{Ratio} & \textbf{Method} & \textbf{PPL} & \textbf{Val Loss} & \textbf{sec. / iter.} \\
      \midrule
      \textcolor{gray}{100\%} & \textcolor{gray}{LLaMA-2-7B}
      & $\textcolor{gray}{{5.53}}$
      & $\textcolor{gray}{{6.25}}$
      & $\textcolor{gray}{{-}}$ \\
      \midrule

      & Dobi-SVD & 6.95 & \underline{7.17} & 75.09 \\
      80\% & Our ($\varepsilon = 0.1$) & \underline{6.90} & 7.19 & 55.62 \\
      \cellcolor{mypink!10}
      & \cellcolor{mypink!10}Our ($\varepsilon = 1.0$)
      & \cellcolor{mypink!10}\bfseries 6.87
      & \cellcolor{mypink!10}\bfseries 7.11
      & \cellcolor{mypink!10}55.62 \\
      & Our ($\varepsilon = 10$) & 6.93 & \underline{7.17} & 55.62 \\
      \midrule

      & Dobi-SVD & 7.47 & 7.83 & 75.09 \\
      \cellcolor{mypink!10}60\%
      & \cellcolor{mypink!10}Our ($\varepsilon = 0.1$)
      & \cellcolor{mypink!10}\bfseries 7.33
      & \cellcolor{mypink!10}\bfseries 7.62
      & \cellcolor{mypink!10}55.62 \\
      & Our ($\varepsilon = 1.0$) & \underline{7.34} & \underline{7.63} & 55.62 \\
      & Our ($\varepsilon = 10$) & \underline{7.34} & 7.72 & 55.62 \\
      \midrule

      & Dobi-SVD & 13.34 & 14.72 & 75.09 \\
      40\% & Our ($\varepsilon = 0.1$) & \underline{12.54} & \underline{13.56} & 55.62 \\
      \cellcolor{mypink!10}
      & \cellcolor{mypink!10}Our ($\varepsilon = 1.0$)
      & \cellcolor{mypink!10}\bfseries 11.40
      & \cellcolor{mypink!10}\bfseries 12.46
      & \cellcolor{mypink!10}55.62 \\
      & Our ($\varepsilon = 10$) & 13.25 & 14.70 & 55.62 \\
      \bottomrule
    \end{tabular}%
    \label{tab:llama}
\end{table}

Table~\ref{tab:llama} reports the results. We compare perplexity (PPL), validation loss, and seconds per iteration. At the compression ratio of $80\%$, all methods remain close, but our method with $\varepsilon=1.0$ achieves the best PPL and validation loss. At $60\%$, the best two variants are $\varepsilon=0.1$ and $\varepsilon=1.0$, both improving over Dobi-SVD, which suggests that the proposed backward pass remains useful beyond the low-compression regime. The advantage becomes most pronounced at the aggressive $40\%$ compression ratio: our method with $\varepsilon=1.0$ reduces PPL from $13.34$ to $11.40$ and validation loss from $14.72$ to $12.46$. This indicates that stable gradients are especially important when the rank-selection problem becomes harder. The ablation also shows that the regularization strength matters. The value $\varepsilon=1.0$ gives the best overall trade-off, whereas stronger regularization degrades quality. Finally, our pipeline yields about a $1.35\times$ speedup.

\section{Conclusion}
\label{section:conclusion}
In conclusion, we proposed an efficient GPU-friendly SVD computation via the polar decomposition, utilizing an iterative scheme for the polar factor that is dominated by matrix multiplications. This structure allows for additional speedups by performing computations in mixed precision while retaining practical accuracy.

For the backward pass, we propose to differentiate the SVD through the stable differentiation of the polar decomposition and the symmetric eigenvalue decomposition.
Finally, we provide a perturbation-based perspective that makes the backward computation robust for perturbed inputs and nearly degenerate spectra (including singular values close to zero), and show how regularization allows us to control this behavior.

\bibliography{iclr2027_conference}

@article{grishina2025accelerating,
  title={Accelerating Newton-Schulz Iteration for Orthogonalization via Chebyshev-type Polynomials},
  author={Grishina, Ekaterina and Smirnov, Matvey and Rakhuba, Maxim},
  journal={arXiv preprint arXiv:2506.10935},
  year={2025}
}

@inproceedings{
amsel2025polar,
title={The Polar Express: Optimal Matrix Sign Methods and their Application to the Muon Algorithm},
author={Noah Amsel and David Persson and Christopher Musco and Robert M. Gower},
booktitle={The Fourteenth International Conference on Learning Representations},
year={2026},
url={https://openreview.net/forum?id=yRtgZ1K8hO}
}

@article{nakatsukasa2010optimizing,
  title={Optimizing Halley's iteration for computing the matrix polar decomposition},
  author={Nakatsukasa, Yuji and Bai, Zhaojun and Gygi, Fran{\c{c}}ois},
  journal={SIAM Journal on Matrix Analysis and Applications},
  volume={31},
  number={5},
  pages={2700--2720},
  year={2010},
  publisher={SIAM}
}

@misc{jordan2024muon,
  author       = {Keller Jordan and Yuchen Jin and Vlado Boza and Jiacheng You and
                  Franz Cesista and Laker Newhouse and Jeremy Bernstein},
  title        = {Muon: An optimizer for hidden layers in neural networks},
  year         = {2024},
  url          = {https://kellerjordan.github.io/posts/muon/}
}

@misc{higham2015faster,
  author       = {Nick Higham},
  title        = {Faster SVD via Polar Decomposition},
  year         = {2015},
  url          = {https://nhigham.com/2015/11/24/faster-svd-via-polar-decomposition/}
}

@article{bjorck1971iterative,
  title={An iterative algorithm for computing the best estimate of an orthogonal matrix},
  author={Bj{\"o}rck, {\AA}ke and Bowie, Clazett},
  journal={SIAM Journal on Numerical Analysis},
  volume={8},
  number={2},
  pages={358--364},
  year={1971},
  publisher={SIAM}
}

@article{kovarik1970some,
  title={Some iterative methods for improving orthonormality},
  author={Kovarik, Zdislav},
  journal={SIAM Journal on Numerical Analysis},
  volume={7},
  number={3},
  pages={386--389},
  year={1970},
  publisher={SIAM}
}

@article{nakatsukasa2013stable,
  title={Stable and efficient spectral divide and conquer algorithms for the symmetric eigenvalue decomposition and the SVD},
  author={Nakatsukasa, Yuji and Higham, Nicholas J},
  journal={SIAM Journal on Scientific Computing},
  volume={35},
  number={3},
  pages={A1325--A1349},
  year={2013},
  publisher={SIAM}
}

@article{struski2024efficient,
  title={Efficient GPU implementation of randomized SVD and its applications},
  author={Struski, {\L}ukasz and Morkisz, Pawe{\l} and Spurek, Przemys{\l}aw and Bernabeu, Samuel Rodriguez and Trzci{\'n}ski, Tomasz},
  journal={Expert Systems with Applications},
  volume={248},
  pages={123462},
  year={2024},
  publisher={Elsevier}
}

@article{demmel1990accurate,
  title={Accurate singular values of bidiagonal matrices},
  author={Demmel, James and Kahan, William},
  journal={SIAM Journal on Scientific and Statistical Computing},
  volume={11},
  number={5},
  pages={873--912},
  year={1990},
  publisher={SIAM}
}

@book{golub2013matrix,
  title={Matrix computations},
  author={Golub, Gene H and Van Loan, Charles F},
  year={2013},
  publisher={JHU press}
}

@article{wang2020deep,
  title={Deep CNNs meet global covariance pooling: Better representation and generalization},
  author={Wang, Qilong and Xie, Jiangtao and Zuo, Wangmeng and Zhang, Lei and Li, Peihua},
  journal={IEEE transactions on pattern analysis and machine intelligence},
  volume={43},
  number={8},
  pages={2582--2597},
  year={2020},
  publisher={IEEE}
}

@article{oseledets2011tensor,
  title={Tensor-train decomposition},
  author={Oseledets, Ivan V},
  journal={SIAM Journal on Scientific Computing},
  volume={33},
  number={5},
  pages={2295--2317},
  year={2011},
  publisher={SIAM}
}

@inproceedings{wang2025svdllm,
  title={{SVD}-{LLM}: Truncation-aware Singular Value Decomposition for Large Language Model Compression},
  author={Xin Wang and Yu Zheng and Zhongwei Wan and Mi Zhang},
  booktitle={The Thirteenth International Conference on Learning Representations},
  year={2025}
}

@article{meng2024pissa,
  title={Pissa: Principal singular values and singular vectors adaptation of large language models},
  author={Meng, Fanxu and Wang, Zhaohui and Zhang, Muhan},
  journal={Advances in Neural Information Processing Systems},
  volume={37},
  pages={121038--121072},
  year={2024}
}

@article{abdi2010principal,
  title={Principal component analysis},
  author={Abdi, Herv{\'e} and Williams, Lynne J},
  journal={Wiley interdisciplinary reviews: computational statistics},
  volume={2},
  number={4},
  pages={433--459},
  year={2010},
  publisher={Wiley Online Library}
}

@inproceedings{cremonesi2010performance,
  title={Performance of recommender algorithms on top-n recommendation tasks},
  author={Cremonesi, Paolo and Koren, Yehuda and Turrin, Roberto},
  booktitle={Proceedings of the fourth ACM conference on Recommender systems},
  pages={39--46},
  year={2010}
}

@article{rodpysh2023employing,
  title={Employing singular value decomposition and similarity criteria for alleviating cold start and sparse data in context-aware recommender systems},
  author={Rodpysh, Keyvan Vahidy and Mirabedini, Seyed Javad and Banirostam, Touraj},
  journal={Electronic Commerce Research},
  volume={23},
  number={2},
  pages={681--707},
  year={2023},
  publisher={Springer}
}

@article{yuan2019singular,
  title={Singular value decomposition based recommendation using imputed data},
  author={Yuan, Xiaofeng and Han, Lixin and Qian, Subin and Xu, Guoxia and Yan, Hong},
  journal={Knowledge-Based Systems},
  volume={163},
  pages={485--494},
  year={2019},
  publisher={Elsevier}
}

@inproceedings{shi2024domain,
  title={Domain generalization via nuclear norm regularization},
  author={Shi, Zhenmei and Ming, Yifei and Fan, Ying and Sala, Frederic and Liang, Yingyu},
  booktitle={Conference on Parsimony and Learning},
  pages={179--201},
  year={2024},
  organization={PMLR}
}

@book{anderson1999lapack,
  title={LAPACK users' guide},
  author={Anderson, Edward and Bai, Zhaojun and Bischof, Christian and Blackford, L Susan and Demmel, James and Dongarra, Jack and Du Croz, Jeremy and Greenbaum, Anne and Hammarling, Sven and McKenney, Alan and others},
  year={1999},
  publisher={SIAM}
}

@article{paszke2019pytorch,
  title={PyTorch: An imperative style, high-performance deep learning library},
  author={Paszke, Adam and Gross, Sam and Massa, Francisco and Lerer, Adam and Bradbury, James and Chanan, Gregory and Killeen, Trevor and Lin, Zeming and Gimelshein, Natalia and Antiga, Luca and others},
  journal={Advances in neural information processing systems},
  volume={32},
  year={2019}
}

@manual{nvidia_cusolver_pdf,
  title={{cuSOLVER Library Reference Guide (PDF)}},
  author={NVIDIA},
  organization={NVIDIA},
  year={2025},
  url={https://docs.nvidia.com/cuda/pdf/CUSOLVER_Library.pdf},
}

@inproceedings{song2022fast,
  title={Fast Differentiable Matrix Square Root},
  author={Yue Song and Nicu Sebe and Wei Wang},
  booktitle={International Conference on Learning Representations},
  year={2022}
}

@InProceedings{giles2008extended,
author="Giles, Mike B.",
title="Collected Matrix Derivative Results for Forward and Reverse Mode Algorithmic Differentiation",
booktitle="Advances in Automatic Differentiation",
year="2008",
publisher="Springer Berlin Heidelberg",
pages="35--44"
}

@techreport{townsend2016differentiating,
  title={Differentiating the singular value decomposition},
  author={Townsend, James},
  year={2016},
  institution={Technical Report 2016, https://j-towns.github.io/papers/svd-derivative.pdf}
}

@article{parkina2025coala,
  title={COALA: Numerically Stable and Efficient Framework for Context-Aware Low-Rank Approximation},
  author={Parkina, Uliana and Rakhuba, Maxim},
  journal={Advances in Neural Information Processing Systems},
  volume={39},
  year={2025}
}

@article{higham1994parallel,
  title={A parallel algorithm for computing the polar decomposition},
  author={Higham, Nicholas J and Papadimitriou, Pythagoras},
  journal={Parallel computing},
  volume={20},
  number={8},
  pages={1161--1173},
  year={1994},
  publisher={Elsevier}
}

@article{hestenes1958inversion,
  title={Inversion of matrices by biorthogonalization and related results},
  author={Hestenes, Magnus R},
  journal={Journal of the Society for Industrial and Applied Mathematics},
  volume={6},
  number={1},
  pages={51--90},
  year={1958},
  publisher={SIAM}
}

@article{de1989one,
  title={A one-sided Jacobi algorithm for computing the singular value decomposition on a vector computer},
  author={de Rijk, P\_P M\_},
  journal={SIAM journal on scientific and statistical computing},
  volume={10},
  number={2},
  pages={359--371},
  year={1989},
  publisher={SIAM}
}

@book{demmel1997applied,
  title={Applied numerical linear algebra},
  author={Demmel, James W},
  year={1997},
  publisher={SIAM}
}

@article{bright2025matrix,
  title={Matrix calculus for machine learning and beyond},
  author={Edelman, Alan and Johnson, Steven G},
  journal={Massachusetts Institute of Technology: MIT OpenCourseWare},
  year={2022}
}

@book{Tyrtyshnikov2025MatrixAnalysis,
  author    = {Tyrtyshnikov, Eugene E.},
  title     = {Matrichnyi analiz i osnovy algebry [Matrix Analysis and Fundamentals of Algebra]},
  year      = {2025},
  publisher = {MCNMO},
  note = {[in Russian]},
  language  = {russian}
}

@article{zhang2024differentiable,
  title={Differentiable SVD based on Moore-Penrose Pseudoinverse for Inverse Imaging Problems},
  author={Zhang, Yinghao and Hu, Yue},
  journal={arXiv preprint arXiv:2411.14141},
  year={2024}
}

@article{chen2014stable,
  title={A stable scaling of Newton-Schulz for improving the sign function computation of a Hermitian matrix},
  author={Chen, Jie and Chow, Edmond},
  journal={Preprint ANL/MCS-P5059-0114 (https://www.mcs.anl.gov/papers/P5059-0114.pdf)},
  year={2014}
}

@inproceedings{ionescu2015matrix,
  title={Matrix backpropagation for deep networks with structured layers},
  author={Ionescu, Catalin and Vantzos, Orestis and Sminchisescu, Cristian},
  booktitle={Proceedings of the IEEE international conference on computer vision},
  pages={2965--2973},
  year={2015}
}

@article{wang2021robust,
  title={Robust differentiable SVD},
  author={Wang, Wei and Dang, Zheng and Hu, Yinlin and Fua, Pascal and Salzmann, Mathieu},
  journal={IEEE transactions on pattern analysis and machine intelligence},
  volume={44},
  number={9},
  pages={5472--5487},
  year={2021},
  publisher={IEEE}
}

@article{kanchi2025differentiable,
  title={Differentiable singular value decomposition (SVD)},
  author={Kanchi, Rohit Sunil and He, Sicheng},
  journal={Mechanical Systems and Signal Processing},
  volume={237},
  pages={112817},
  year={2025},
  publisher={Elsevier}
}

@book{bhatia2013matrix,
  title={Matrix analysis},
  author={Bhatia, Rajendra},
  volume={169},
  year={2013},
  publisher={Springer Science \& Business Media}
}

@article{gawlik2016computing,
  title={Iterative computation of the Fr{\'e}chet derivative of the polar decomposition},
  author={Gawlik, Evan S and Leok, Melvin},
  journal={SIAM Journal on Matrix Analysis and Applications},
  volume={38},
  number={4},
  pages={1354--1379},
  year={2017},
  publisher={SIAM}
}

@article{noferini2017generalizedfunction,
author = {Noferini, Vanni},
title = {A Formula for the Fréchet Derivative of a Generalized Matrix Function},
journal = {SIAM Journal on Matrix Analysis and Applications},
volume = {38},
number = {2},
pages = {434-457},
year = {2017},
doi = {10.1137/16M1072851},
URL = {https://doi.org/10.1137/16M1072851},
eprint = {https://doi.org/10.1137/16M1072851}
}

@article{nakatsukasa2012backward,
  title={Backward stability of iterations for computing the polar decomposition},
  author={Nakatsukasa, Yuji and Higham, Nicholas J},
  journal={SIAM Journal on Matrix Analysis and Applications},
  volume={33},
  number={2},
  pages={460--479},
  year={2012},
  publisher={SIAM}
}

@article{vannieuwenhoven2012truncation,
author = {Vannieuwenhoven, Nick and Vandebril, Raf and Meerbergen, Karl},
title = {A New Truncation Strategy for the Higher-Order Singular Value Decomposition},
journal = {SIAM Journal on Scientific Computing},
volume = {34},
number = {2},
pages = {A1027-A1052},
year = {2012},
doi = {10.1137/110836067},
URL = {https://doi.org/10.1137/110836067},
eprint = {https://doi.org/10.1137/110836067}
}

@article{wang2018atomo,
  title={Atomo: Communication-efficient learning via atomic sparsification},
  author={Wang, Hongyi and Sievert, Scott and Liu, Shengchao and Charles, Zachary and Papailiopoulos, Dimitris and Wright, Stephen},
  journal={Advances in neural information processing systems},
  volume={31},
  year={2018}
}

@article{jolliffe2016principal,
    author = {Jolliffe, Ian T. and Cadima, Jorge},
    title = {Principal component analysis: a review and recent developments},
    journal = {Philosophical Transactions of the Royal Society A: Mathematical, Physical and Engineering Sciences},
    volume = {374},
    number = {2065},
    pages = {20150202},
    year = {2016},
    month = {04},
    issn = {1364-503X},
    doi = {10.1098/rsta.2015.0202},
    url = {https://doi.org/10.1098/rsta.2015.0202},
    eprint = {https://royalsocietypublishing.org/rsta/article-pdf/doi/10.1098/rsta.2015.0202/1381479/rsta.2015.0202.pdf},
}

@inproceedings{wang2017global,
  author={Wang, Qilong and Li, Peihua and Zhang, Lei},
  booktitle={2017 IEEE Conference on Computer Vision and Pattern Recognition (CVPR)}, 
  title={G2DeNet: Global Gaussian Distribution Embedding Network and Its Application to Visual Recognition}, 
  year={2017},
  volume={},
  number={},
  pages={6507-6516},
  doi={10.1109/CVPR.2017.689}
}

@inproceedings{
qinsi2025dobisvd,
title={Dobi-{SVD}: Differentiable {SVD} for {LLM} Compression and Some New Perspectives},
author={Wang Qinsi and Jinghan Ke and Masayoshi Tomizuka and Kurt Keutzer and Chenfeng Xu},
booktitle={The Thirteenth International Conference on Learning Representations},
year={2025},
url={https://openreview.net/forum?id=kws76i5XB8}
}

@software{jax2018github,
  author = {James Bradbury and Roy Frostig and Peter Hawkins and Matthew James Johnson and Yash Katariya and Chris Leary and Dougal Maclaurin and George Necula and Adam Paszke and Jake Vander{P}las and Skye Wanderman-{M}ilne and Qiao Zhang},
  title = {{JAX}: composable transformations of {P}ython+{N}um{P}y programs},
  url = {http://github.com/jax-ml/jax},
  version = {0.3.13},
  year = {2018},
}

@book{hansen2010discrete,
  title={Discrete inverse problems: insight and algorithms},
  author={Hansen, Per Christian},
  year={2010},
  publisher={SIAM}
}

@article{touvron2023llama,
  title={Llama 2: Open foundation and fine-tuned chat models},
  author={Touvron, Hugo and Martin, Louis and Stone, Kevin and Albert, Peter and Almahairi, Amjad and Babaei, Yasmine and Bashlykov, Nikolay and Batra, Soumya and Bhargava, Prajjwal and Bhosale, Shruti and others},
  journal={arXiv preprint arXiv:2307.09288},
  year={2023}
}

@article{merity2016pointer,
  title={Pointer sentinel mixture models},
  author={Merity, Stephen and Xiong, Caiming and Bradbury, James and Socher, Richard},
  journal={arXiv preprint arXiv:1609.07843},
  year={2016}
}

@article{Baydin2018automatic,
  author  = {Atilim Gunes Baydin and Barak A. Pearlmutter and Alexey Andreyevich Radul and Jeffrey Mark Siskind},
  title   = {Automatic Differentiation in Machine Learning: a Survey},
  journal = {Journal of Machine Learning Research},
  year    = {2018},
  volume  = {18},
  number  = {153},
  pages   = {1--43},
  url     = {http://jmlr.org/papers/v18/17-468.html}
}

@book{griewank2008evaluating,
  title={Evaluating derivatives: principles and techniques of algorithmic differentiation},
  author={Griewank, Andreas and Walther, Andrea},
  year={2008},
  publisher={SIAM}
}
\bibliographystyle{iclr2027_conference}

\appendix
\section{Implementation Details}
\label{appendix:impl_details}

\subsection{Newton-Schulz Initialization Normalization}
\label{appendix:impl_details:ns_init}

An important detail in implementing the Newton--Schulz iteration is that it converges only if $\|X_0\|_2 < \sqrt{3}$. Therefore, the initial matrix should be normalized before the iteration begins. Ideally, one would divide the matrix by its spectral norm so that $\|X_0\|_2 = 1$, but computing the spectral norm is expensive. The standard approach is to normalize the matrix using the Frobenius norm; here we use a normalization method based on the $1$-norm and $\infty$-norm, as in the QDWH implementation in JAX\footnote{\url{https://github.com/jax-ml/jax/blob/23db456a8acd01a04ed5a9f87f8265cc21703926/jax/_src/tpu/linalg/qdwh.py\#L132}}:

$$
X_0 = M / \sqrt{\| M \|_1 \| M \|_{\infty}}, \quad \text{since}\quad
\| M \|_2 \leq \sqrt{\| M \|_1 \| M \|_{\infty}}.
$$

The $1$-norm and $\infty$-norm are straightforward to compute, as they correspond to the maximum $\ell_1$ norm of the columns and rows of the matrix, respectively.

\begin{algorithm}[ht]
\caption{Normalization for Newton-Schulz based iterations}
\label{alg:normalization}
\begin{algorithmic}[1]
\REQUIRE
Matrix $M \in \mathbb{R}^{m \times n}$ ($m \geq n$)
\ENSURE
Normalized matrix $\hat{M}$
\STATE $a \gets \|M\|_1 \|M \|_{\infty}$
\STATE $\hat{M} \gets M / \sqrt{a}$
\RETURN $\hat{M}$
\end{algorithmic}
\end{algorithm}

\subsection{Handling Rank-Deficient Cases}
\label{appendix:impl_details:rank_deficient}

Importantly, CANS and QDWH iterations do not converge to an orthogonal polar factor when applied to a rank-deficient matrix $M$ with rank $r < n$. Instead, both iterations converge to a factor
\begin{equation*}
    W_r = U \begin{pmatrix}
    I_r & 0\\
    0 & 0\\
    \end{pmatrix} V^\top = \begin{pmatrix}
    U_r & 0
    \end{pmatrix} V^\top,
\end{equation*}
where $U$ and $V$ are the left and right singular factors of $M$ respectively. 
Therefore, the matrix $W_rV$ does not have orthonormal columns, as required in the definition of the SVD. This issue can be resolved by orthogonalizing the matrix $W_rV$ using a QR factorization~\cite[Section 5.5]{nakatsukasa2013stable}.
Importantly, the QR routine used in this procedure should return an $R$ factor with non-negative diagonal entries.
In practice, the rank of a matrix cannot be computed stably, so rank-deficient cases can be detected by checking the column norms of $WV$. See the experiments for rank-deficient matrices in Appendix~\ref{appendix:additional_exps:low_rank}.

\subsection{Matrix Generation Scheme}
\label{appendix:impl_details:gen}

We generate a random matrix $M \in \mathbb{R}^{n \times m}$ with $\kappa_2(M) = c$ and $\operatorname{rank}(M) = r$ using the following scheme. First, we generate two matrices $G_1 \in \mathbb{R}^{n \times m}$ and $G_2 \in \mathbb{R}^{m \times m}$ with entries sampled independently from the standard normal distribution. Then, we perform thin QR factorizations:
\begin{equation*}
G_1 = Q_1 R_1, \quad G_2 = Q_2 R_2.
\end{equation*}
Finally, we obtain the matrix
\begin{equation*}
M = Q_1 \Sigma Q_2,
\end{equation*}
where $\Sigma \in \mathbb{R}^{m \times m}$ is a diagonal matrix with diagonal entries $\Sigma_{ii} = c^{(r - i) / (r - 1)}$ for $i \leq r$, and $\Sigma_{ii} = 0$ otherwise.

We generate a random matrix $M \in \mathbb{R}^{n \times n}$ with spectral gap $\Delta$ using the same scheme. First, we generate two matrices $G_1 \in \mathbb{R}^{n \times n}$ and $G_2 \in \mathbb{R}^{n \times n}$ with entries sampled independently from the standard normal distribution and perform QR factorizations. Finally, the matrix is obtained via:
\begin{equation*}
M = Q_1 \Sigma Q_2,
\end{equation*}
where $\Sigma \in \mathbb{R}^{n \times n}$ is a diagonal matrix with diagonal entries $\Sigma_{ii} = i$ for $i < n$, and $\Sigma_{nn} = n - 1 + \Delta$.

\subsection{Reproducibility Details}
\label{appendix:impl_details:repro}

All experiments were conducted on an NVIDIA B200 GPU with 192 GB of VRAM, running Debian Linux 12. Each individual experiment took at most one hour, and the total compute budget did not exceed one GPU day.

For the CANS iteration, we use a polynomial of degree $3$ with a limit of $50$ iterations. For preprocessing, we use $2$ iterations, and the parameter $\delta$ is set to $0.99$, as in the original implementation~\cite{grishina2025accelerating}. As a stopping criterion, we use a tolerance of $10^{-5}$ for fp32 experiments and $10^{-3}$ for tf32 experiments.

For CANS SVD, we use $\varepsilon_{\texttt{QR}} = 10^{-5}$ for fp32 experiments and $10^{-3}$ for tf32 experiments to detect rank-deficient cases.

\section{Additional Experiments}
\label{appendix:additional_exps}

\subsection{Scaling QR Factorization and Matrix Multiplications on Modern GPUs}
\label{appendix:additional_exps:scaling}

Recent advances in GPU engineering, driven by the growth of deep learning popularity, have made matrix multiplication operations very fast on GPUs. We empirically compare how the performance of matrix multiplication scales relative to QR decomposition. We provide a comparison across several NVIDIA GPU architectures: A100, H100, H200, and B200. As shown in Figure~\ref{fig:exp:algo_time}, QR factorization does not substantially speed up beyond the A100 architecture, whereas matrix multiplications in tf32 precision continue to speed up with newer generations of GPUs. Moreover, Figure~\ref{fig:exp:qr_mm} shows that matrix multiplications can be much faster than QR factorization across various matrix sizes. This performance gap becomes larger with newer GPU generations.

\begin{figure}[ht]
    \centering
    \begin{subfigure}[t]{0.48\textwidth}
        \includegraphics[width=\linewidth]{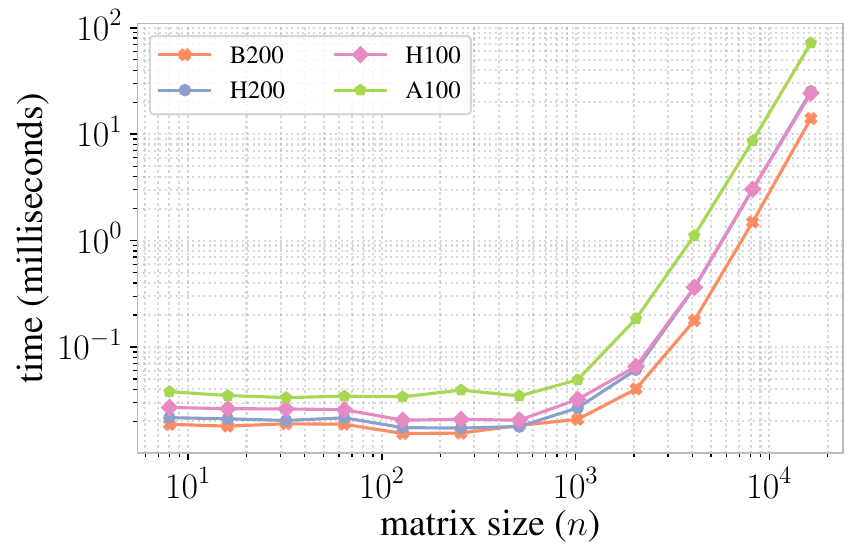}
        \caption{Matrix multiplication computation in tf32 precision}
        \label{fig:exp:algo_time:mm_tf32}
    \end{subfigure}
    ~
    \begin{subfigure}[t]{0.48\textwidth}
        \includegraphics[width=\linewidth]{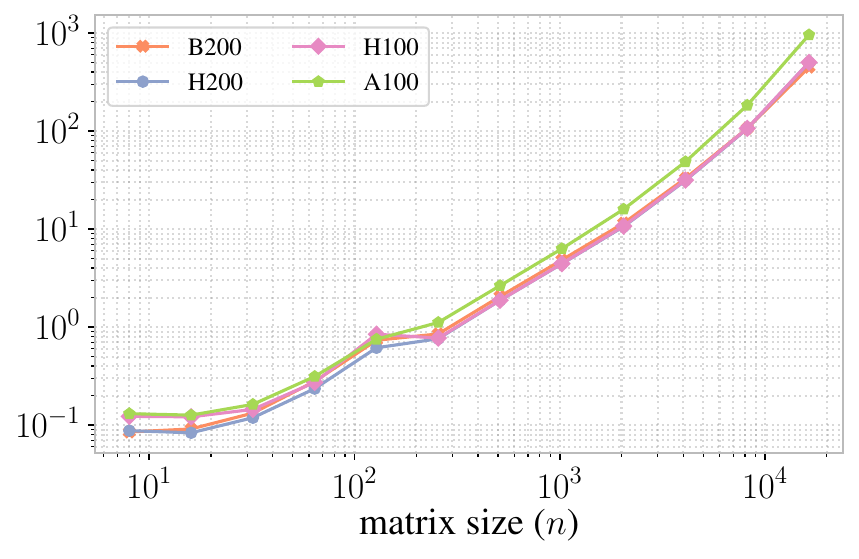}
        \caption{QR decomposition computation}
        \label{fig:exp:algo_time:qr}
    \end{subfigure}
    \caption{Comparison of computation time for matrix multiplications in tf32 precision (left plot) and QR decompositions (right plot) across different GPUs on random square matrices of various sizes. The total wall-clock time is averaged over $100$ trials.}
    \label{fig:exp:algo_time}
\end{figure}

\begin{figure}[ht]
    \centering
    \begin{subfigure}[t]{0.48\textwidth}
        \includegraphics[width=\linewidth]{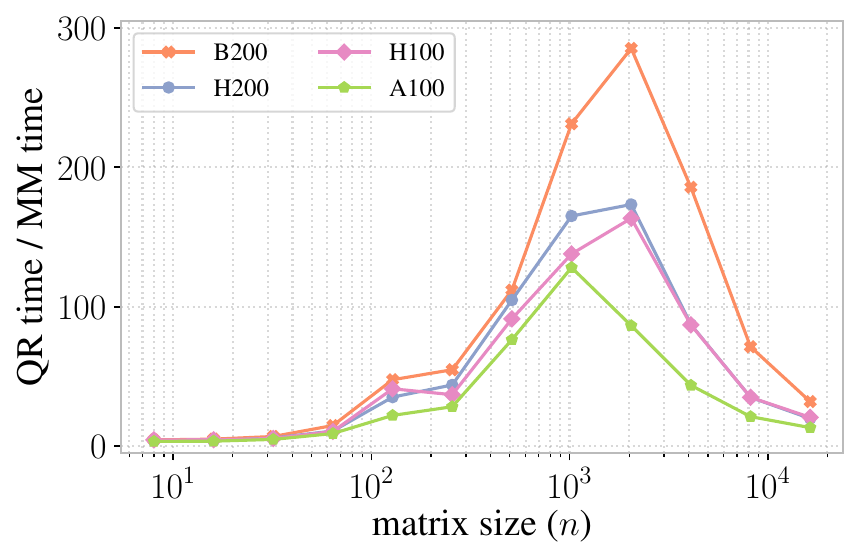}
        \caption{Matrix multiplications performed in tf32}
        \label{fig:exp:qr_mm:tf32}
    \end{subfigure}
    ~
    \begin{subfigure}[t]{0.48\textwidth}
        \includegraphics[width=\linewidth]{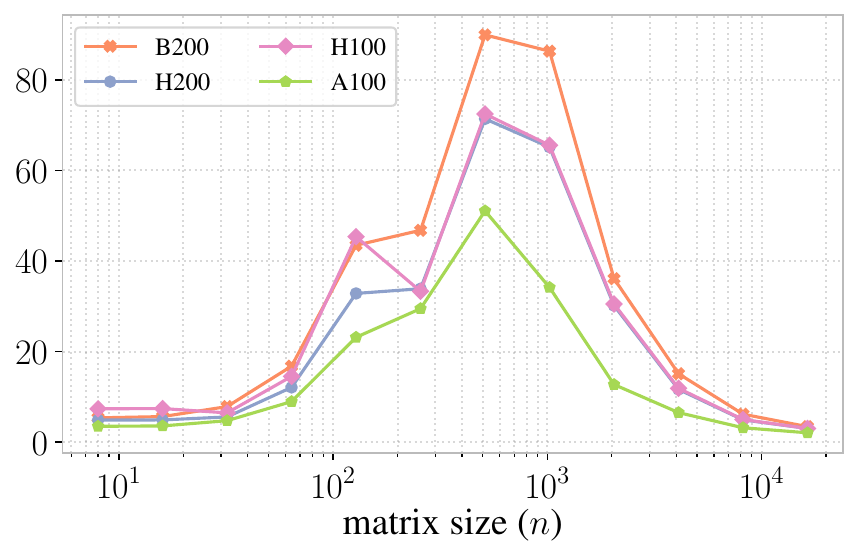}
        \caption{Matrix multiplications performed in fp32}
        \label{fig:exp:qr_mm:fp32}
    \end{subfigure}
    \caption{Comparison illustrating how many matrix multiplications can be executed within the runtime of a single QR decomposition for square matrices of different sizes across different GPUs. In the left plot, matrix multiplications are performed in tf32 precision. In the right plot, matrix multiplications are performed in fp32 precision. The total wall-clock time is averaged over $100$ trials.}
    \label{fig:exp:qr_mm}
\end{figure}

\subsection{Low Rank Matrices}
\label{appendix:additional_exps:low_rank}

We provide additional experiments comparing the performance and accuracy of different algorithms on low-rank square random matrices of varying sizes. Tables~\ref{tab:forward:recon_error_low} and~\ref{tab:forward:ortho_error_low}, as well as Figure~\ref{fig:exp:forward:low_rank}, present the results. The CUDA POLAR algorithm failed to compute the SVD even for matrices with rank one less than full rank. The other methods demonstrate stability across various rank values. Notably, CANS SVD in fp32 precision runs slower than QDWH SVD on large matrices compared to the full-rank case, where they run in approximately the same time.

\begin{figure}[ht]
    \centering
    \includegraphics[width=0.48\linewidth]{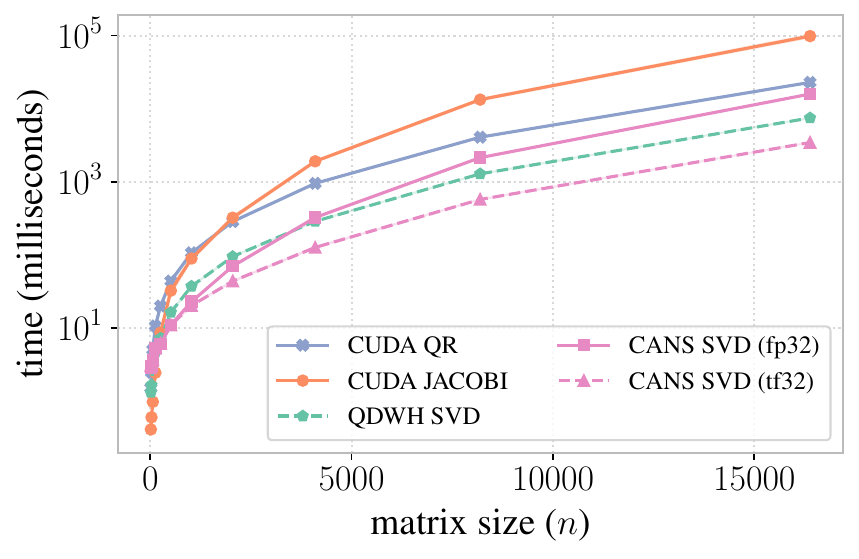}
    \caption{Comparison of SVD algorithms on a random low rank matrix $M \in \mathbb{R}^{n \times n}$ with $\operatorname{rank} (M) = \lfloor n /10 \rfloor$. Total wall-clock time was averaged over $100$ trials and reported with the standard deviation. \textbf{Note:} we do not report CUDA POLAR for this setting, since it has substantially large reconstruction error. And algorithms have different accuracy, see Tables~\ref{tab:forward:recon_error_low} and \ref{tab:forward:ortho_error_low} for a more comprehensive analysis. The experiments are conducted on an NVIDIA B200 GPU.}
    \label{fig:exp:forward:low_rank}
\end{figure}

\begin{table}[ht]
    \centering
    \small
    \caption{Comparison of SVD algorithms accuracy on a random square matrix $M \in \mathbb{R}^{4096 \times 4096}$ for different ranks. Reported relative reconstruction error $\|M - U\Sigma V^\top\|_F / \|M\|_F$, with standard deviation. Metrics were averaged over $100$ sampled matrices. \colorbox{red!10}{Red} indicates cases with significantly large error.}
    \begin{tabular}{l|ccccc}
        \toprule
        Method & $1$ & $64$ & $1024$ & $4095$ & $4096$ \\
        \midrule
        CANS SVD (fp32) & $1.9_{\color{gray}\scriptscriptstyle\pm .0} \cdot 10^{-6}$ & $2.7_{\color{gray}\scriptscriptstyle\pm .1} \cdot 10^{-6}$ & $4.3_{\color{gray}\scriptscriptstyle\pm .0} \cdot 10^{-6}$ & $8.4_{\color{gray}\scriptscriptstyle\pm .4} \cdot 10^{-6}$ & $4.9_{\color{gray}\scriptscriptstyle\pm .0} \cdot 10^{-6}$ \\
        CANS SVD (tf32) & $5.3_{\color{gray}\scriptscriptstyle\pm .2} \cdot 10^{-4}$ & $7.5_{\color{gray}\scriptscriptstyle\pm .2} \cdot 10^{-4}$ & $1.1_{\color{gray}\scriptscriptstyle\pm .0} \cdot 10^{-3}$ & $9.8_{\color{gray}\scriptscriptstyle\pm .0} \cdot 10^{-4}$ & $9.7_{\color{gray}\scriptscriptstyle\pm .0} \cdot 10^{-4}$ \\
        QDWH SVD & $1.7_{\color{gray}\scriptscriptstyle\pm .1} \cdot 10^{-6}$ & $1.9_{\color{gray}\scriptscriptstyle\pm .0} \cdot 10^{-6}$ & $2.7_{\color{gray}\scriptscriptstyle\pm .0} \cdot 10^{-6}$ & $3.5_{\color{gray}\scriptscriptstyle\pm .0} \cdot 10^{-6}$ & $3.4_{\color{gray}\scriptscriptstyle\pm .0} \cdot 10^{-6}$ \\
        CUDA POLAR & \colorbox{red!10}{$8.3_{\color{gray}\scriptscriptstyle\pm 22.7} \cdot 10^{0}$} & \colorbox{red!10}{$7.6_{\color{gray}\scriptscriptstyle\pm .0} \cdot 10^{1}$} & \colorbox{red!10}{$7.6_{\color{gray}\scriptscriptstyle\pm .0} \cdot 10^{1}$} & \colorbox{red!10}{$7.6_{\color{gray}\scriptscriptstyle\pm .0} \cdot 10^{1}$} & $4.3_{\color{gray}\scriptscriptstyle\pm .0} \cdot 10^{-6}$ \\
        CUDA QR & $7.4_{\color{gray}\scriptscriptstyle\pm .5} \cdot 10^{-7}$ & $1.3_{\color{gray}\scriptscriptstyle\pm .2} \cdot 10^{-6}$ & $6.6_{\color{gray}\scriptscriptstyle\pm .4} \cdot 10^{-6}$ & $1.8_{\color{gray}\scriptscriptstyle\pm .1} \cdot 10^{-5}$ & $1.8_{\color{gray}\scriptscriptstyle\pm .1} \cdot 10^{-5}$ \\
        CUDA JACOBI & $3.5_{\color{gray}\scriptscriptstyle\pm .6} \cdot 10^{-5}$ & $9.5_{\color{gray}\scriptscriptstyle\pm .2} \cdot 10^{-5}$ & $5.1_{\color{gray}\scriptscriptstyle\pm .4} \cdot 10^{-4}$ & $1.0_{\color{gray}\scriptscriptstyle\pm .2} \cdot 10^{-4}$ & $9.7_{\color{gray}\scriptscriptstyle\pm .4} \cdot 10^{-4}$ \\
        \bottomrule
    \end{tabular}
    \label{tab:forward:recon_error_low}
\end{table}

\begin{table}[ht]
    \centering
    \small
    \caption{Comparison of SVD algorithms accuracy on a random square matrix $M \in \mathbb{R}^{4096 \times 4096}$ for different ranks. Reported relative orthogonality $\max\{\|U^\top U - I\|_F / \|I\|_F, \|V^\top V - I\|_F / \|I\|_F \}$, with standard deviation. Metrics were averaged over $100$ sampled matrices. }
    \begin{tabular}{l|ccccc}
        \toprule
        Method & $1$ & $64$ & $1024$ & $4095$ & $4096$ \\
        \midrule
        CANS SVD (fp32) & $1.6_{\color{gray}\scriptscriptstyle\pm .0} \cdot 10^{-6}$ & $1.5_{\color{gray}\scriptscriptstyle\pm .0} \cdot 10^{-6}$ & $1.8_{\color{gray}\scriptscriptstyle\pm .0} \cdot 10^{-6}$ & $3.1_{\color{gray}\scriptscriptstyle\pm 1.2} \cdot 10^{-6}$ & $3.7_{\color{gray}\scriptscriptstyle\pm .0} \cdot 10^{-6}$ \\
        CANS SVD (tf32) & $1.9_{\color{gray}\scriptscriptstyle\pm .2} \cdot 10^{-6}$ & $1.9_{\color{gray}\scriptscriptstyle\pm .0} \cdot 10^{-6}$ & $2.2_{\color{gray}\scriptscriptstyle\pm .0} \cdot 10^{-6}$ & $3.0_{\color{gray}\scriptscriptstyle\pm .0} \cdot 10^{-6}$ & $6.6_{\color{gray}\scriptscriptstyle\pm .0} \cdot 10^{-4}$ \\
        QDWH SVD & $1.6_{\color{gray}\scriptscriptstyle\pm .0} \cdot 10^{-6}$ & $1.5_{\color{gray}\scriptscriptstyle\pm .0} \cdot 10^{-6}$ & $1.8_{\color{gray}\scriptscriptstyle\pm .0} \cdot 10^{-6}$ & $3.6_{\color{gray}\scriptscriptstyle\pm .0} \cdot 10^{-6}$ & $3.6_{\color{gray}\scriptscriptstyle\pm .0} \cdot 10^{-6}$ \\
        CUDA POLAR & $7.5_{\color{gray}\scriptscriptstyle\pm 1.4} \cdot 10^{-6}$ & $3.2_{\color{gray}\scriptscriptstyle\pm .0} \cdot 10^{-6}$ & $6.9_{\color{gray}\scriptscriptstyle\pm .0} \cdot 10^{-6}$ & $1.2_{\color{gray}\scriptscriptstyle\pm .0} \cdot 10^{-5}$ & $4.1_{\color{gray}\scriptscriptstyle\pm .0} \cdot 10^{-6}$ \\
        CUDA QR & $7.9_{\color{gray}\scriptscriptstyle\pm .2} \cdot 10^{-6}$ & $7.7_{\color{gray}\scriptscriptstyle\pm .2} \cdot 10^{-6}$ & $6.4_{\color{gray}\scriptscriptstyle\pm .2} \cdot 10^{-6}$ & $9.6_{\color{gray}\scriptscriptstyle\pm .3} \cdot 10^{-6}$ & $9.6_{\color{gray}\scriptscriptstyle\pm .2} \cdot 10^{-6}$ \\
        CUDA JACOBI & $2.5_{\color{gray}\scriptscriptstyle\pm .2} \cdot 10^{-6}$ & $3.8_{\color{gray}\scriptscriptstyle\pm .1} \cdot 10^{-4}$ & $1.7_{\color{gray}\scriptscriptstyle\pm .1} \cdot 10^{-3}$ & $2.1_{\color{gray}\scriptscriptstyle\pm .0} \cdot 10^{-3}$ & $2.1_{\color{gray}\scriptscriptstyle\pm .0} \cdot 10^{-3}$ \\
        \bottomrule
    \end{tabular}
    \label{tab:forward:ortho_error_low}
\end{table}

\subsection{Comparison with Additional Baselines}
\label{appendix:additional_exps:gram}

For a more comprehensive ablation study, we compare CANS SVD with additional baselines in terms of accuracy and performance. The additional baselines are listed below:
\begin{itemize}
\item \textbf{EVD of Gram Matrix.} A naive way of computing SVD via EVD of the Gram matrix $M^\top M$.
\item \textbf{EVD of Extended Matrix.} A way of computing SVD via EVD of the extended matrix $\begin{pmatrix}
    0 & M \\ M^\top & 0
\end{pmatrix}$.
\item \textbf{SciPy SVD.} Computing SVD on CPU using the SciPy routine to double-check the correctness of the proposed method.
\end{itemize}
We expect the Gram-based algorithm to lose accuracy due to its numerical instability (see Remark~\ref{remark:demmel}), while the extended-matrix-based algorithm should provide better accuracy but be much slower. The SciPy routine should be accurate, but it is computed on CPU and therefore is much slower, so we do not report its time.

Results are presented in Table~\ref{appendix:tab:forward:gram} and Figure~\ref{fig:exp:forward:gram}. As anticipated, the SciPy routine, Extended EVD, and CANS SVD in single precision have approximately the same accuracy, while the Gram-based algorithm starts losing accuracy even when $\kappa_2(M) = 100$ due to numerical instability and has a large reconstruction error for $\kappa_2(M) = 10^3$. Extended EVD is slower than the Gram-based method and CANS SVD.

\begin{table}[ht!]
    \centering
    \small
    \caption{Comparison of the accuracy of SVD algorithms on a random square matrix $M \in \mathbb{R}^{4096 \times 4096}$ for different condition numbers. The reported metric is the relative reconstruction error $\|M - U\Sigma V^\top\|_F / \|M\|_F$, with standard deviation. Results are averaged over $100$ sampled matrices.}
    \begin{tabular}{l|ccccc}
        \toprule
        Method & $10$ & $10^{2}$ & $10^{3}$ & $10^{4}$ & $10^{6}$ \\
        \midrule
        CANS SVD (fp32) & $4.9_{\color{gray}\scriptscriptstyle\pm .0} \cdot 10^{-6}$ & $4.2_{\color{gray}\scriptscriptstyle\pm .0} \cdot 10^{-6}$ & $4.3_{\color{gray}\scriptscriptstyle\pm .1} \cdot 10^{-6}$ & $4.4_{\color{gray}\scriptscriptstyle\pm .0} \cdot 10^{-6}$ & $5.3_{\color{gray}\scriptscriptstyle\pm .1} \cdot 10^{-6}$ \\
        Extended & $5.8_{\color{gray}\scriptscriptstyle\pm .0} \cdot 10^{-6}$ & $5.6_{\color{gray}\scriptscriptstyle\pm .1} \cdot 10^{-6}$ & $5.4_{\color{gray}\scriptscriptstyle\pm .1} \cdot 10^{-6}$ & $1.1_{\color{gray}\scriptscriptstyle\pm .0} \cdot 10^{-5}$ & $9.0_{\color{gray}\scriptscriptstyle\pm .0} \cdot 10^{-6}$ \\
        SciPy & $5.2_{\color{gray}\scriptscriptstyle\pm .0} \cdot 10^{-6}$ & $5.1_{\color{gray}\scriptscriptstyle\pm .0} \cdot 10^{-6}$ & $5.2_{\color{gray}\scriptscriptstyle\pm .0} \cdot 10^{-6}$ & $5.2_{\color{gray}\scriptscriptstyle\pm .0} \cdot 10^{-6}$ & $5.1_{\color{gray}\scriptscriptstyle\pm .0} \cdot 10^{-6}$ \\
        Gram & $4.4_{\color{gray}\scriptscriptstyle\pm .0} \cdot 10^{-6}$ & $2.3_{\color{gray}\scriptscriptstyle\pm .0} \cdot 10^{-5}$ & $1.1_{\color{gray}\scriptscriptstyle\pm .0} \cdot 10^{-4}$ & $3.9_{\color{gray}\scriptscriptstyle\pm .0} \cdot 10^{-4}$ & $6.9_{\color{gray}\scriptscriptstyle\pm .0} \cdot 10^{-4}$ \\
        \bottomrule
    \end{tabular}
    \label{appendix:tab:forward:gram}
\end{table}

\begin{figure}[ht]
    \centering
    \includegraphics[width=0.48\linewidth]{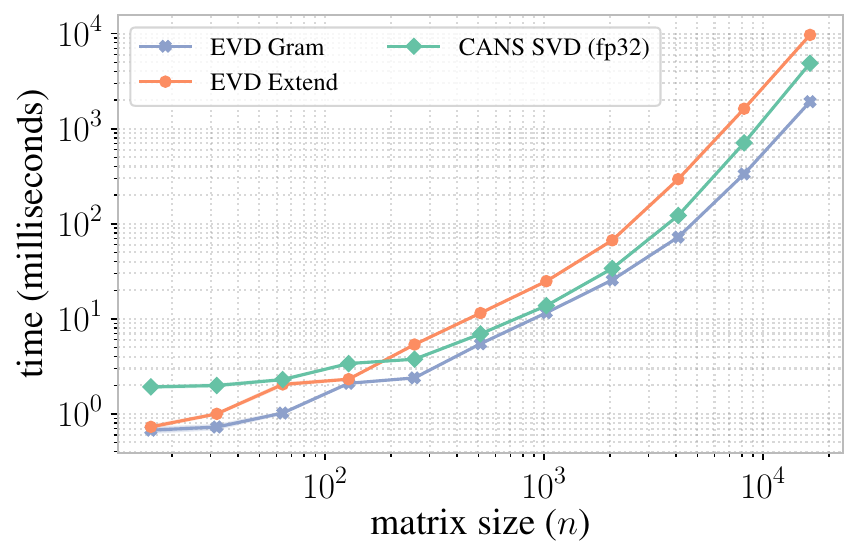}
    \caption{Comparison of SVD algorithms on a square random matrices $M \in \mathbb{R}^{n \times n}$ with $\kappa_2(M) = 10$. Total wall-clock time was averaged over $100$ trials and reported with the standard deviation. The experiments are conducted on an NVIDIA B200 GPU.}
    \label{fig:exp:forward:gram}
\end{figure}

\subsection{Ablation on Iterations for Computing Polar Factor}
\label{appendix:additional_exps:polar_factor}

In this section, we provide an ablation study of different methods for computing the polar factor that rely heavily on matrix multiplications. For a more difficult setting, we test all methods on a random square matrix $M \in \mathbb{R}^{1024\times 1024}$ with a large condition number $\kappa_2(M) = 10^6$.

Results are shown in Figure~\ref{fig:exp:forward:ablation}. For the Muon Newton--Schulz iteration~\cite{jordan2024muon} and Polar-Express~\cite{amsel2025polar}, we use the default parameters provided in the authors implementations. As expected, the Muon iteration does not converge to an orthogonal matrix and yields a high relative error. The original Newton--Schulz iteration converges slowly, whereas Polar-Express and CANS converge substantially faster and exhibit similar convergence behavior. However, Polar-Express reaches a plateau at a higher error level than CANS and only drops to a comparable accuracy at the final iteration. This may be caused by the fact that the input to the polynomial is divided by $1.01$ at all iterations except the final one. This behavior makes it difficult to use a stopping criterion for this method. CANS does not have this issue, and we therefore choose it as our main iteration.

\begin{figure}[ht]
    \centering
    \includegraphics[width=0.8\linewidth]{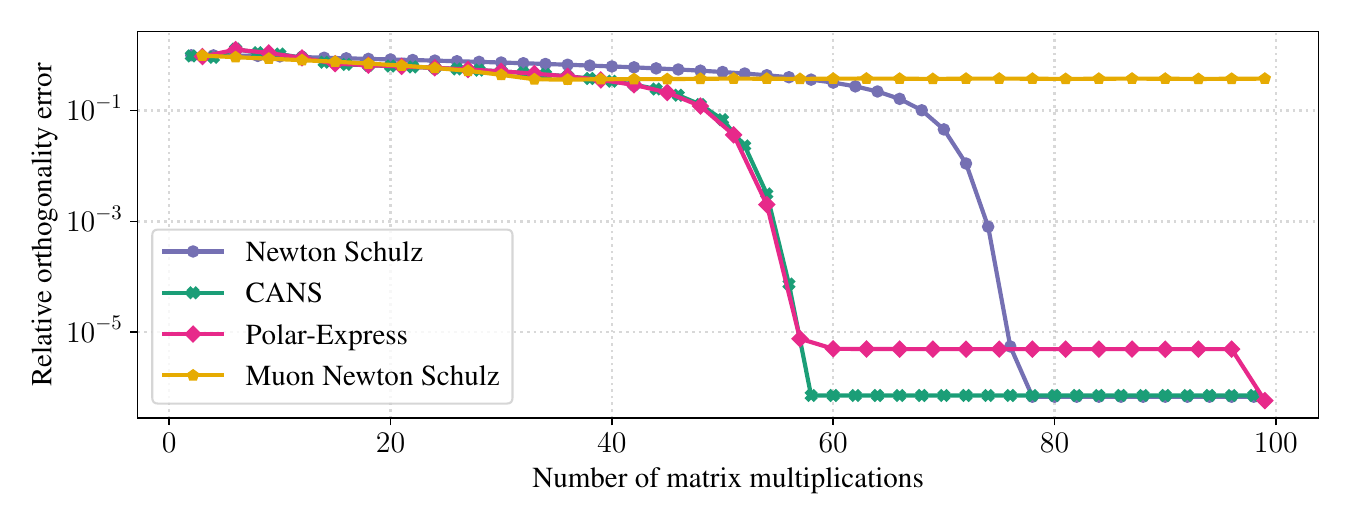}
    \caption{Comparison of different methods for computing polar decomposition on a square random matrices $M \in \mathbb{R}^{n \times n}$ with $\kappa_2(M) = 10^6$. The reported metric is relative orthogonality error $\| W^\top W - I \|_F / \|I\|_F$.}
    \label{fig:exp:forward:ablation}
\end{figure}

\subsection{CUDA JACOBI Hyperparameters}
\label{appendix:additional_exps:jacobi}

According to the cuSOLVER documentation\footnote{\url{https://docs.nvidia.com/cuda/cusolver/index.html}}, CUDA JACOBI provides two hyperparameters controlling accuracy: \texttt{tolerance} and \texttt{max\_sweeps}. We observe that the default value of the \texttt{max\_sweeps} parameter is sufficient to reach the required precision, as stated in the documentation. Moreover, the number of sweeps does not change substantially across different settings.

We test different values of \texttt{tolerance} in Figure~\ref{fig:exp:forward:jacobi}. Decreasing the \texttt{tolerance} parameter makes the algorithm more accurate but slower. For our experiments, we choose the default parameter value, which reaches approximately the same accuracy as CANS SVD in mixed precision. Making this algorithm more accurate would make it slower and allow it to outperform our method in fewer cases, making the comparison unfair.

\begin{figure}[ht]
    \centering
    \includegraphics[width=0.48\linewidth]{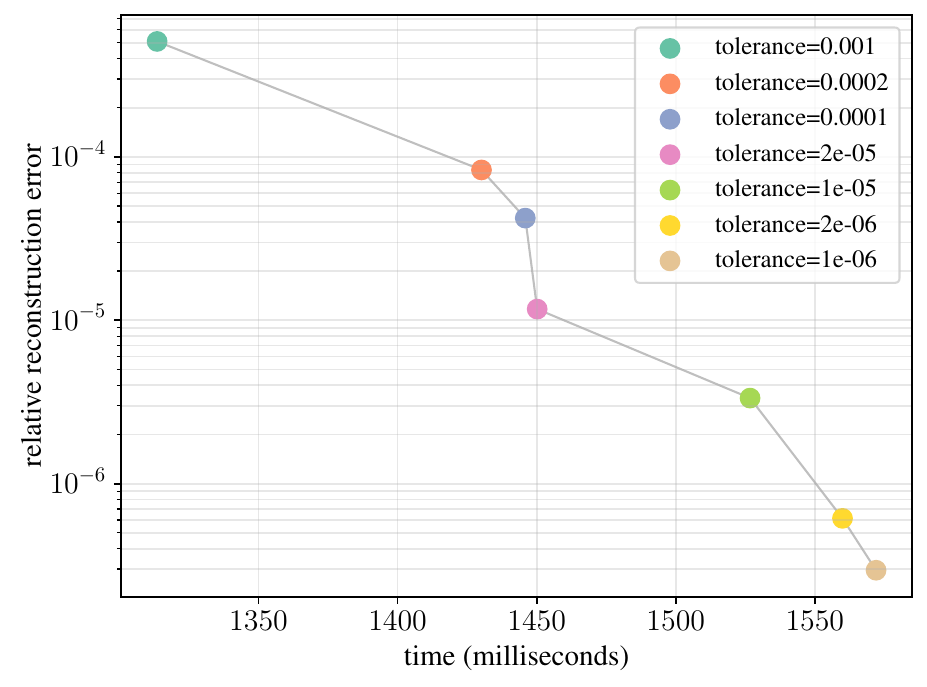}
    \caption{Pareto front for different CUDA JACOBI tolerance settings on square random matrices $M \in \mathbb{R}^{4096 \times 4096}$ with $\kappa_2(M) = 10^2$.}
    \label{fig:exp:forward:jacobi}
\end{figure}

\subsection{Choosing \texorpdfstring{$\delta$}{delta} Hyperparameter}
\label{appendix:additional_exps:delta}

Here, we provide an ablation study on the hyperparameter $\delta$ in the CANS algorithm~\cite{grishina2025accelerating}.

Let us briefly introduce how $\delta$ appears in CANS. In a nutshell, CANS algorithm divided into two steps: preprocessing step, and main CANS iteration. In the preprocessing stage, we apply a polynomial that finds the largest interval $[t,1]$ whose values can be mapped into the interval $[1-\delta,1+\delta]$. Then, CANS iteration is applied with parameters $a = 1-\delta$ and $b = 1+\delta$. For singular values inside the interval $[a,b]$, the iteration has accelerated convergence (see Corollary 3.5 of~\cite{grishina2025accelerating}). For singular values smaller than $a$, one can also show convergence, moreover, it is faster than the classical NS iteration.

In Tables~\ref{tab:delta_ablation:fp32} and~\ref{tab:delta_ablation:tf32}, the reported relative reconstruction error depends on different $\delta$ values and condition numbers. The hyperparameter $\delta$ provides a tradeoff between accuracy and speed: increasing $\delta$ enlarges the interval of accelerated convergence, but also increases the stability constant, which results in reduced numerical accuracy. In our paper, we chose $\delta = 0.99$ based on empirical results across different arithmetic precisions. We could have made our method even more accurate by choosing a smaller value of $\delta$. However, we argue that selected setting is superior as it allows us to achieve accuracy comparable to the methods used by default libraries, while running faster.

\begin{table}[ht]
\centering
\caption{Ablation on different $\delta$ for preprocessing on a random matrix $M \in \mathbb{R}^{4096 \times 4096}$ for different condition numbers. The reported metric is the relative reconstruction error of CANS SVD (fp32).}
\label{tab:delta_ablation:fp32}
\begin{tabular}{c|cccc}
\toprule
$\kappa_2(M) \downarrow \ \delta \rightarrow$ & $0.9$ & $0.99$ & $0.999$ & $0.9999$ \\
\midrule
$10^2$ & $3.30 \cdot 10^{-6}$ & $4.26 \cdot 10^{-6}$ & $7.25 \cdot 10^{-6}$ & $1.84 \cdot 10^{-5}$ \\
$10^4$ & $3.77 \cdot 10^{-6}$ & $4.41 \cdot 10^{-6}$ & $1.05 \cdot 10^{-5}$ & $2.51 \cdot 10^{-5}$ \\
$10^6$ & $4.00 \cdot 10^{-6}$ & $5.38 \cdot 10^{-6}$ & $1.69 \cdot 10^{-5}$ & $4.88 \cdot 10^{-5}$ \\
\bottomrule
\end{tabular}
\end{table}

\begin{table}[ht]
\centering
\caption{Ablation on different $\delta$ for preprocessing on a random matrix $M \in \mathbb{R}^{4096 \times 4096}$ for different condition numbers. The reported metric is the relative reconstruction error of CANS SVD (tf32).}
\label{tab:delta_ablation:tf32}
\begin{tabular}{c|cccc}
\toprule
$\kappa_2(M) \downarrow \ \delta \rightarrow$ & $0.9$ & $0.99$ & $0.999$ & $0.9999$ \\
\midrule
$10$ & $7.20 \cdot 10^{-4}$ & $9.73 \cdot 10^{-4}$ & $2.98 \cdot 10^{-3}$ & $9.15 \cdot 10^{-3}$ \\
$10^2$ & $6.57 \cdot 10^{-4}$ & $8.24 \cdot 10^{-4}$ & $1.84 \cdot 10^{-3}$ & $4.33 \cdot 10^{-3}$ \\
$10^3$ & $7.04 \cdot 10^{-4}$ & $8.52 \cdot 10^{-4}$ & $1.91 \cdot 10^{-3}$ & $3.63 \cdot 10^{-3}$ \\
\bottomrule
\end{tabular}
\end{table}

\subsection{Difference Between TF32 and BF16 Matrix Multiplications}
\label{appendix:additional_exps:tf32_vs_bf16}

We compare matrix multiplication in bf16 and tf32 precision on an NVIDIA B200 GPU. Figure~\ref{fig:exp:tf32_vs_bf16} presents the results of this comparison. In Figure~\ref{fig:exp:tf32_vs_bf16:mm}, the wall-clock time of the matrix multiplications themselves is compared; we observe a negligible speedup when using bf16 precision compared to tf32. In Figure~\ref{fig:exp:tf32_vs_bf16:svd}, the wall-clock time of the SVD computation with different precisions is shown, and the effect of using bf16 precision is almost invisible.

\begin{figure}[ht]
    \centering
    \begin{subfigure}[t]{0.48\textwidth}
        \includegraphics[width=\linewidth]{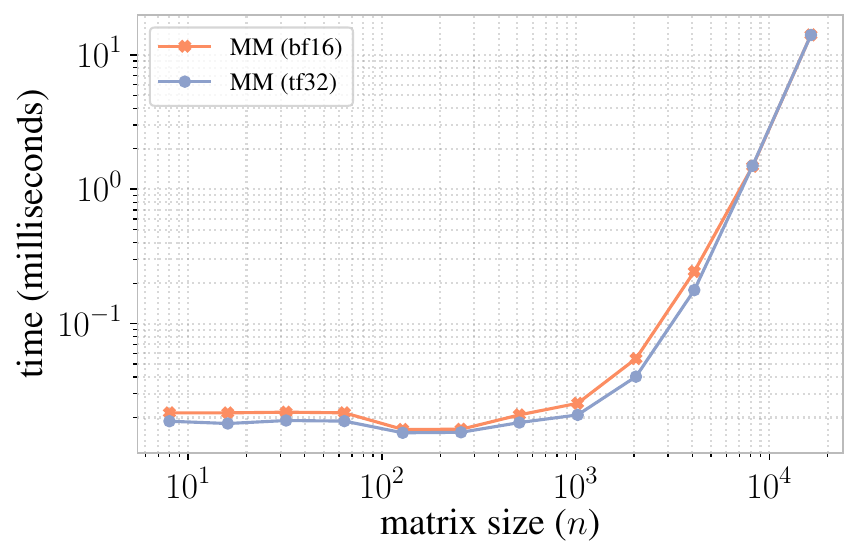}
        \caption{Matrix multiplications time}
        \label{fig:exp:tf32_vs_bf16:mm}
    \end{subfigure}
    ~
    \begin{subfigure}[t]{0.48\textwidth}
        \includegraphics[width=\linewidth]{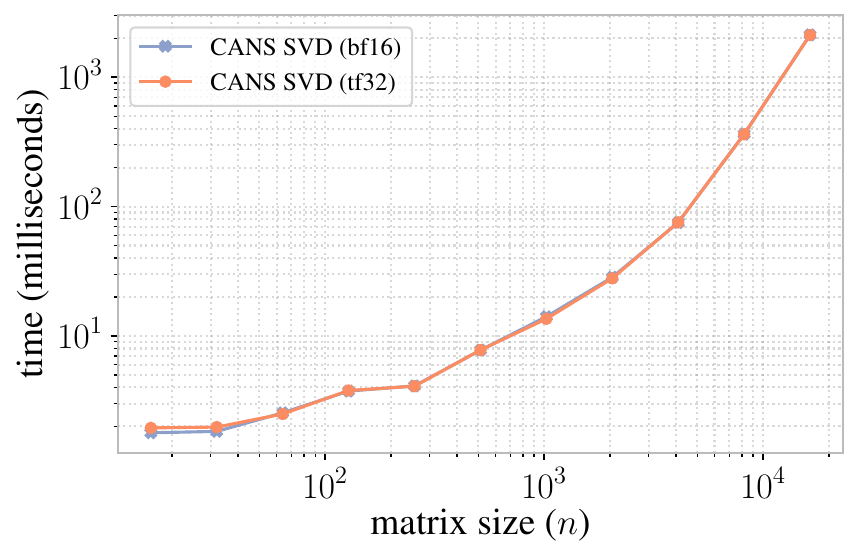}
        \caption{SVD computation time}
        \label{fig:exp:tf32_vs_bf16:svd}
    \end{subfigure}
    \caption{Performance comparison of different precisions (tf32 and bf16) for matrix multiplication. The left plot shows the runtime of the matrix multiplications themselves. The right plot shows how the SVD runtime changes when the multiplication precision is varied. The total wall-clock time is averaged over $100$ trials. Experiments are conducted on an NVIDIA B200 GPU.}
    \label{fig:exp:tf32_vs_bf16}
\end{figure}

\subsection{Performance Comparison of Methods for Solving Symmetric Lyapunov Equation}
\label{appendix:additional_exps:eigh_vs_iter}

Here, we conduct an ablation study of different methods for solving the symmetric Lyapunov equation
\begin{equation*}
HX + XH = M,
\end{equation*}
where $M$ is a symmetric matrix and $H$ is symmetric positive semidefinite. We compare the iterative method proposed in \cite{song2022fast} with a closed-form solution based on symmetric EVD (EIGH). Figure~\ref{fig:exp:additional_exps:lyapunov} summarizes the results. For small ($n \leq 64$) and large ($n > 1024$) matrices, the EIGH-based approach is faster than the iterative method. Moreover, as the matrix size $n$ increases, the EIGH-based approach becomes increasingly more efficient. Additionally, we observe that the convergence of the iterative method depends on the condition number of the matrix, whereas the EIGH-based approach does not.

\begin{figure}[ht]
    \centering
    \begin{subfigure}[t]{0.48\textwidth}
        \includegraphics[width=\linewidth]{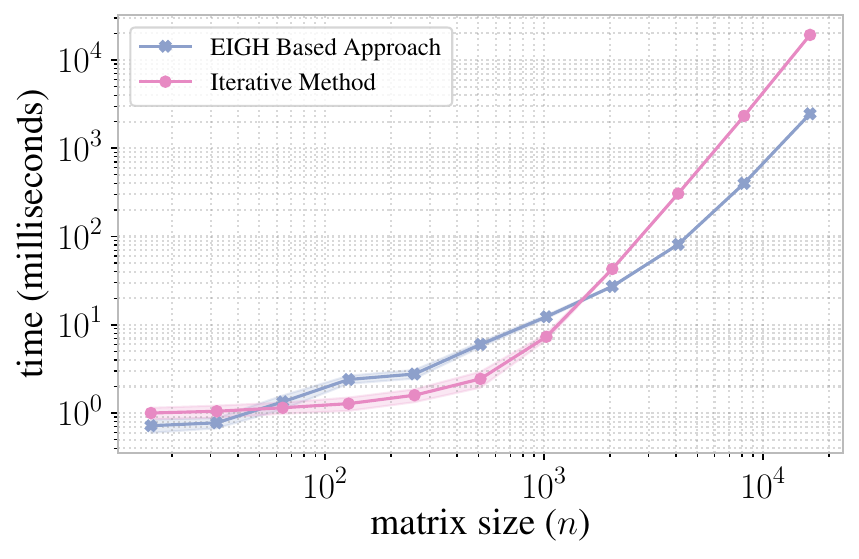}
        \caption{Comparison performed on random square matrices with $\kappa_2(M) = \kappa_2(H) = 100$ and varying size $n$.}
        \label{fig:exp:additional_exps:lyapunov:size}
    \end{subfigure}
    ~
    \begin{subfigure}[t]{0.48\textwidth}
        \includegraphics[width=\linewidth]{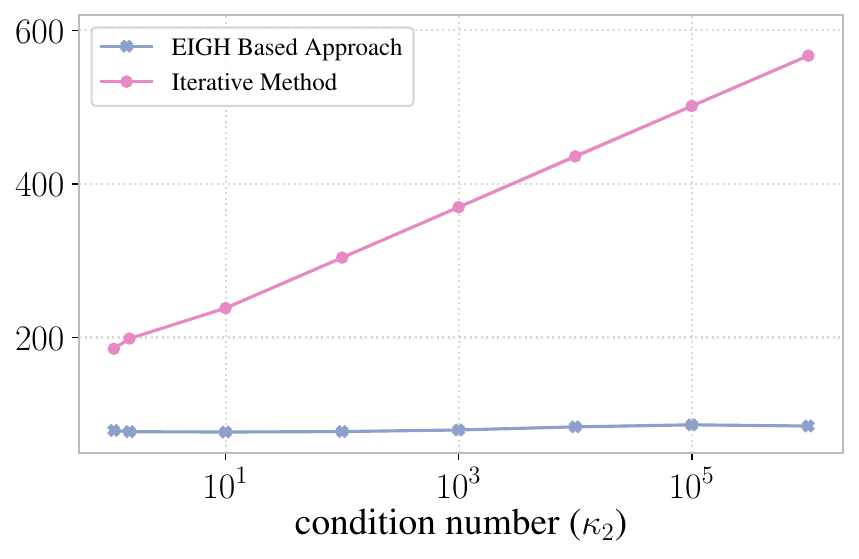}
        \caption{Comparison performed on random square matrices $M, H \in \mathbb{R}^{4096 \times 4096}$ with varying condition number $\kappa_2(M) = \kappa_2(H)$.}
        \label{fig:exp:additional_exps:lyapunov:cond}
    \end{subfigure}
    \caption{Comparison of different methods for solving the symmetric Lyapunov equation under different scenarios. The total wall-clock time was averaged over $100$ trials and reported with the standard deviation. The experiments were conducted on an NVIDIA B200 GPU.}
    \label{fig:exp:additional_exps:lyapunov}
\end{figure}

\section{\texorpdfstring{SVD-inv~\cite{zhang2024differentiable} inconsistency}{SVD-inv inconsistency}}
\label{appendix:example:explanation}

The SVD-inv work~\cite{zhang2024differentiable} contains a bug in Equation 25 in Theorem 3.1: the pseudoinverse of the matrix $\begin{pmatrix}
\sigma & -\sigma \\
-\sigma & \sigma \\
\end{pmatrix}$ is not equal to $\begin{pmatrix}
1 / 2\sigma & -1 / 2\sigma \\
0 & 0 \\
\end{pmatrix}$.

\begin{example}[SVD-inv~\cite{zhang2024differentiable} inconsistency]
\label{example:inconsistency}
    Consider the SVD of the identity matrix, $I = U \Sigma V^\top$. At this degenerate point (all singular values are equal), the SVD-inv backward expression of~\cite{zhang2024differentiable} exhibits inconsistency. Despite the fact that $U = V$ at this point, distinct gradients are returned during the backward pass even in symmetric cases. If $\overline{U} = \overline{\Sigma} = 0$, then the formula gives
    $$
    \overline{I} = 0.
    $$
    If $\overline{V} = \overline{\Sigma} = 0$, then the same formula gives
    $$
    \overline{I} = U (U^\top \overline{U})_{\text{off-diag}} V^\top.
    $$
This inconsistency indicates that, at degenerate points, the SVD-inv backward pass may return gradients that do not correspond to the expected gradient of the underlying matrix-valued function.
\end{example}

\begin{explanation}[Explanation for Example~\ref{example:inconsistency}]
    For the first case, setting $\overline{U} = \overline{\Sigma} = 0$ and using $VV^\top = I$ in the final formulas of~\citep[Theorem 3.2]{zhang2024differentiable} gives:
    \begin{align*}
        \overline{I} &= U\left.[ \cancel{(F \odot \left.[ U^\top \overline{U} - \overline{U}^\top U\right.])} \Sigma + \cancel{T \odot (U^\top \overline{U})} \right.] V^\top \\
        &+ \cancel{(I - UU^\top) \overline{U} \Sigma^{-1} V^\top} \\
        &+ \cancel{U (I \odot \overline{\Sigma}) V^\top} \\
        &+ U \Sigma (F \odot \left.[ V^\top \overline{V} - \overline{V}^\top V\right.]) V^\top \\
        &+ U \Sigma^{-1} \overline{V}^\top \cancel{(I - VV^\top)} \\
        &=
        U \Sigma (F \odot \left.[ V^\top \overline{V} - \overline{V}^\top V\right.]) V^\top = 0,
    \end{align*}
    since all singular values of the identity matrix are equal and thus $F = 0$.

    For the second case, setting $\overline{V} = \overline{\Sigma} = 0$, using $UU^\top = I$, and setting $F = 0$ gives:
    \begin{align*}
        \overline{I} &= U\left.[ \cancel{(F \odot \left.[ U^\top \overline{U} - \overline{U}^\top U\right.])} \Sigma + T \odot (U^\top \overline{U}) \right.] V^\top \\
        &+ \cancel{(I - UU^\top)} \overline{U} \Sigma^{-1} V^\top \\
        &+ \cancel{U (I \odot \overline{\Sigma}) V^\top} \\
        &+ \cancel{U \Sigma (F \odot \left.[ V^\top \overline{V} - \overline{V}^\top V\right.]) V^\top} \\
        &+ \cancel{U \Sigma^{-1} \overline{V}^\top (I - VV^\top)} \\
        &=
        U (T \odot (U^\top \overline{U})) V^\top = U (U^\top \overline{U})_{\text{off-diag}} V^\top,
    \end{align*}
    since all singular values of $I$ are equal and therefore all elements in $T$ are set to $1$ except the diagonal elements, which are set to $0$.
\end{explanation}

\section{Proofs for Polar Derivatives}
\label{appendix:polar_formuls}

In this section, we present formulas and their proofs for differentiating the polar decomposition, which is used in our method.

Before proving the main formulas, we state the following useful lemma.
\begin{lemma}
\label{lem:lyapunov_identity}
Let $H \in \mathbb{R}^{n \times n}$ be a symmetric positive semi-definite matrix.
Let $Y$ be a matrix satisfying the Lyapunov equation
\begin{equation}
\label{eq:lyap_Y}
H Y + Y H = A,
\end{equation}
where $A \in \mathbb{R}^{n \times n}$ is an arbitrary matrix.
Let $X$ be a solution of
\begin{equation}
\label{eq:lyap_X}
H X + X H = B,
\end{equation}
where $B \in \mathbb{R}^{n \times n}$ is an arbitrary matrix.
Then the following identity holds:
\begin{equation}
\label{eq:lyap_inner}
\langle A, X \rangle
=
\langle Y, B \rangle .
\end{equation}
\end{lemma}

\begin{proof}
We compute
\begin{equation*}
\langle A, X \rangle
=
\langle H Y + Y H, X \rangle
=
\langle H X, Y \rangle + \langle X H, Y \rangle
=
\langle H X + X H, Y \rangle
=
\langle B, Y \rangle,
\end{equation*}
which proves the claim.
\end{proof}

\begin{proposition}
\label{prop:polar_backward}
Let $M \in \mathbb{R}^{m \times n}$ ($m 
\ge n$) be a full-rank matrix with the polar decomposition
\begin{equation*}
M = W H.
\end{equation*}
Let $\mathcal{L}(M) = \ell(W,H)$ be a scalar loss function, and assume that the partial derivatives
\begin{equation*}
\overline{W} := \frac{\partial \ell}{\partial W},
\qquad
\overline{H} := \frac{\partial \ell}{\partial H}
\end{equation*}
are given.
Then the gradient of $\mathcal{L}$ with respect to $M$ is
\begin{equation*}
\frac{\partial \mathcal{L}}{\partial M}
=
\overline{W} H^{-1}
+
2 M X,
\end{equation*}
where $X$ is the symmetric solution of the Lyapunov equation
\begin{equation*}
H X + X H
=
(\overline{H} - W^\top \overline{W} H^{-1})_{\operatorname{sym}}.
\end{equation*}
\end{proposition}

\begin{proof}
Since $M$ is full rank, the polar factor $H$ is invertible.
Moreover, because $H = H^\top$, its differential satisfies
\begin{equation*}
\mathrm{d}H = \mathrm{d}H^\top,
\end{equation*}
i.e., $\mathrm{d}H$ is symmetric.

The differential of the polar factorization $M = W H$ is given by
\begin{equation}
\label{eq:diff_polar}
\mathrm{d}M = \mathrm{d}W\,H + W\,\mathrm{d}H.
\end{equation}
Since $H$ is nonsingular, this relation can be rearranged as
\begin{equation}
\label{eq:diff_polar_w}
\mathrm{d}W = (\mathrm{d}M - W\,\mathrm{d}H)H^{-1}.
\end{equation}

The total differential of the loss function is
\begin{equation}
\label{eq:diff_polar_total_loss}
\mathrm{d}\mathcal{L}
=
\langle \overline{W}, \mathrm{d}W \rangle
+
\langle \overline{H}, \mathrm{d}H \rangle .
\end{equation}

Substituting~\eqref{eq:diff_polar_w} into~\eqref{eq:diff_polar_total_loss} yields
\begin{equation*}
\mathrm{d}\mathcal{L}
=
\left\langle
\overline{W},\,
(\mathrm{d}M - W\,\mathrm{d}H)H^{-1}
\right\rangle
+
\left\langle
\overline{H},\,
\mathrm{d}H
\right\rangle,
\end{equation*}
which can be rewritten as
\begin{equation*}
\mathrm{d}\mathcal{L}
=
\left\langle
\overline{W} H^{-1},\,
\mathrm{d}M
\right\rangle
-
\left\langle
W^\top \overline{W} H^{-1},\,
\mathrm{d}H
\right\rangle
+
\left\langle
\overline{H},\,
\mathrm{d}H
\right\rangle.
\end{equation*}
Therefore,
\begin{equation}
\label{eq:diff_polar_open}
\mathrm{d}\mathcal{L}
=
\left\langle \overline{W} H^{-1}, \mathrm{d}M \right\rangle
+
\left\langle
\overline{H} - W^\top \overline{W} H^{-1},
\mathrm{d}H
\right\rangle.
\end{equation}

Since $H^2 = M^\top M$, differentiation gives the Lyapunov equation
\begin{equation}
\label{eq:diff_polar_part1}
H\,\mathrm{d}H + \mathrm{d}H\,H
=
\mathrm{d}M^\top M + M^\top \mathrm{d}M.
\end{equation}
Let $X$ denote the symmetric solution of
\begin{equation}
\label{eq:diff_polar_part2}
H X + X H
=
(\overline{H} - W^\top \overline{W} H^{-1})_{\operatorname{sym}}.
\end{equation}
Applying Lemma~\ref{lem:lyapunov_identity} to~\eqref{eq:diff_polar_part1} and~\eqref{eq:diff_polar_part2}, we obtain
\begin{equation*}
\left\langle
\overline{H} - W^\top \overline{W} H^{-1},
\mathrm{d}H
\right\rangle
= \left\langle
(\overline{H} - W^\top \overline{W} H^{-1})_{\operatorname{sym}},
\mathrm{d}H
\right\rangle
=
\langle X, \mathrm{d}M^\top M + M^\top \mathrm{d}M \rangle,
\end{equation*}
which simplifies to
\begin{equation}
\label{eq:diff_polar_solve_lyap}
\left\langle
\overline{H} - W^\top \overline{W} H^{-1},
\mathrm{d}H
\right\rangle
=
2 \langle M X, \mathrm{d}M \rangle .
\end{equation}

Combining~\eqref{eq:diff_polar_open} and~\eqref{eq:diff_polar_solve_lyap}, we finally obtain
\begin{equation*}
\mathrm{d}\mathcal{L}
=
\left\langle
\overline{W} H^{-1} + 2 M X,
\mathrm{d}M
\right\rangle,
\end{equation*}
which proves the stated gradient formula.
\end{proof}

So, the task of computing the derivative of the polar decomposition reduces to solving a Lyapunov equation. The following proposition describes how this equation can be solved.

\begin{proposition}[Solving the Symmetric Lyapunov Equation]
\label{prop:lyapunov_eigendecomp}
Let $H \in \mathbb{R}^{n \times n}$ be a symmetric positive definite matrix with eigenvalue decomposition
\begin{equation*}
H = Q \Lambda Q^\top,
\end{equation*}
where $Q$ is orthogonal and $\Lambda = \mathrm{diag}(\lambda_1,\dots,\lambda_n)$ with $\lambda_i > 0$.
For a given matrix $C$, the unique solution $X$ of the Lyapunov equation
\begin{equation*}
H X + X H = C
\end{equation*}
is given by
\begin{equation*}
X = Q \widetilde{X} Q^\top,
\qquad
\widetilde{X}_{ij} = \frac{(Q^\top C Q)_{ij}}{\lambda_i + \lambda_j}.
\end{equation*}
\end{proposition}

\begin{proof}
   In book~\cite{golub2013matrix}, p. 226.
\end{proof}

\begin{remark}
\label{rem:lyapunov_lazy}
If the Eigenvalue Decomposition of the polar factor $H$ is available, solving the Lyapunov equation in the backward pass reduces to simple elementwise operations in the eigenvectors of $H$.
As a result, computing the gradient is almost free compared to the forward pass and does not require additional matrix factorizations.
Moreover, this computation does not depend on spectral gaps, since all terms involve only sums $\lambda_i + \lambda_j$.
\end{remark}

Nevertheless, this equation can also be solved using alternative approaches; for example, the work~\cite{song2022fast} proposes an iterative method. However, in our setting we explicitly seek the SVD via polar decomposition, and therefore the corresponding EVD is already available, which allows us to solve fast this equation.

\begin{proof}[Proof for Theorem~\ref{thm:polar_backward}]
Since $M$ has full column rank, $H$ is symmetric positive definite. By
Proposition~\ref{prop:polar_backward},
\[
\overline M=\overline W H^{-1}+2MX,
\]
where $X$ is the unique symmetric solution of
\[
HX+XH=C.
\]
Applying Proposition~\ref{prop:lyapunov_eigendecomp} to
$H=V\Sigma V^\top$ gives
\[
X
=
V\Bigl(T\odot(V^\top C V)\Bigr)V^\top,
\qquad
T_{ij}=\frac{1}{\sigma_i+\sigma_j}.
\]
Substituting this expression for $X$ into the formula above proves the result.
\end{proof}

\section{Regularized Gradient For Polar Decomposition}
\label{appendix:regularized_polar}

\begin{proposition}[
Tikhonov Regularization for Linear Equations,
{\cite[Sec.~18.16]{Tyrtyshnikov2025MatrixAnalysis}},
{\cite[Sec.~4.5]{hansen2010discrete}}
]
\label{prop:regularization}
Let $A\in\mathbb{R}^{m\times n}$, $b\in\mathbb{R}^m$, and suppose that
\[
\|A-A_\varepsilon\|\leq \varepsilon,
\qquad
\|b-b_\varepsilon\|\leq \varepsilon.
\]
Let $x_0=A^+b$ be the minimum-norm least-squares solution of $Ax=b$, and,
for $\alpha>0$, let
\[
x_{\varepsilon,\alpha}
=
(A_\varepsilon^\top A_\varepsilon+\alpha I)^{-1}
A_\varepsilon^\top b_\varepsilon.
\]
Then, for $0<\varepsilon\leq 1$ and $\alpha=\sqrt{\varepsilon}$,
\[
\|x_{\varepsilon,\alpha}-x_0\|
\leq c\sqrt{\varepsilon},
\]
where $c=c(A,b)>0$ is independent of $\varepsilon$.
\end{proposition}

\begin{lemma}
\label{lemm:module_error}
    For any symmetric $A,B\in\mathbb{R}^{N\times N}$,
\[
\bigl\||A|-|B|\bigr\|_F \;\le\; \|A-B\|_F,
\qquad |A| := (A^\top A)^{1/2}.
\]
\end{lemma}

\begin{proof}
    In book~\cite[Eq.~(X.24)]{bhatia2013matrix}.
\end{proof}

\begin{proof}[Proof of Proposition~\ref{prop:regularized_polar}]
Since $M$ and $\widetilde M$ have full column rank,
\[
M=WH,\qquad
\widetilde M=\widetilde W\widetilde H,
\qquad
H,\widetilde H\succ0.
\]
Throughout the proof, $c>0$ may depend on $M,\overline W,\overline H$, but not on
$\varepsilon$. By smoothness of the polar decomposition in the full-rank regime and
Lemma~\ref{lemm:module_error},
\begin{equation}
\label{eq:regpolar:H_pert}
\|H-\widetilde H\|_F\leq c\varepsilon.
\end{equation}

The matrix $H^{-1}$ is the solution of $HZ=I$, whereas
$\widetilde H_\varepsilon^{-1}$ is the Tikhonov solution of the perturbed system:
\[
(\widetilde H^2+\varepsilon^{1/2}I)\widetilde H_\varepsilon^{-1}
=
\widetilde H.
\]
Applying Proposition~\ref{prop:regularization} columnwise gives
\begin{equation}
\label{eq:regpolar:Hinv_bound}
\|\widetilde H_\varepsilon^{-1}-H^{-1}\|_F
\leq c\varepsilon^{1/2}.
\end{equation}
In particular, $\|\widetilde H_\varepsilon^{-1}\|_2$ is bounded for all sufficiently
small $\varepsilon$.

Define
\[
B=
\left(
\overline H-H^{-1}M^\top\overline W H^{-1}
\right)_{\mathrm{sym}}.
\]
Using the definition of $\widetilde B_\varepsilon$, we add and subtract intermediate
products to obtain
\begin{align*}
&\widetilde H_\varepsilon^{-1}\widetilde M^\top
\overline W\widetilde H_\varepsilon^{-1}
-
H^{-1}M^\top\overline W H^{-1}
\\
&=
(\widetilde H_\varepsilon^{-1}-H^{-1})
\widetilde M^\top\overline W\widetilde H_\varepsilon^{-1}
\\
&\quad+
H^{-1}(\widetilde M^\top-M^\top)
\overline W\widetilde H_\varepsilon^{-1}
\\
&\quad+
H^{-1}M^\top\overline W
(\widetilde H_\varepsilon^{-1}-H^{-1}).
\end{align*}
Since symmetrization does not increase the Frobenius norm,
\eqref{eq:regpolar:Hinv_bound} and
$\|\widetilde M-M\|_F\leq\varepsilon$ imply
\begin{equation}
\label{eq:regpolar:B_bound}
\|\widetilde B_\varepsilon-B\|_F
\leq
c\left(
\|\widetilde H_\varepsilon^{-1}-H^{-1}\|_F
+
\|\widetilde M-M\|_F
\right)
\leq c\varepsilon^{1/2}.
\end{equation}

Let $X$ be the solution of
\[
HX+XH=B.
\]
After vectorization, this equation becomes
\[
\mathcal A\,\mathrm{vec}(X)=\mathrm{vec}(B),
\qquad
\mathcal A:=I\otimes H+H\otimes I.
\]
Similarly, let
\[
\widetilde{\mathcal A}
:=
I\otimes\widetilde H+\widetilde H\otimes I.
\]
By \eqref{eq:regpolar:H_pert} and \eqref{eq:regpolar:B_bound},
\[
\|\mathcal A-\widetilde{\mathcal A}\|
=
\mathcal O(\varepsilon),
\qquad
\|\mathrm{vec}(B)-\mathrm{vec}(\widetilde B_\varepsilon)\|
=
\mathcal O(\varepsilon^{1/2}).
\]
Thus, the effective perturbation level of the Lyapunov system is
$\mathcal O(\varepsilon^{1/2})$. The matrix $\widetilde X_\varepsilon$ defined in the
statement is precisely the Tikhonov solution
\[
\left(
\widetilde{\mathcal A}^{\top}\widetilde{\mathcal A}
+\varepsilon^{1/4}I
\right)
\mathrm{vec}(\widetilde X_\varepsilon)
=
\widetilde{\mathcal A}^{\top}
\mathrm{vec}(\widetilde B_\varepsilon).
\]
Therefore, Proposition~\ref{prop:regularization} yields
\begin{equation}
\label{eq:regpolar:X_bound}
\|\widetilde X_\varepsilon-X\|_F
\leq c\varepsilon^{1/4}.
\end{equation}

By Theorem~\ref{thm:polar_backward},
\[
\overline M
=
\overline W H^{-1}+2MX,
\]
whereas
\[
\overline{\widetilde M}_\varepsilon
=
\overline W\widetilde H_\varepsilon^{-1}
+
2\widetilde M\widetilde X_\varepsilon.
\]
Consequently,
\begin{align*}
\|\overline M-\overline{\widetilde M}_\varepsilon\|_F
&\leq
\|\overline W(H^{-1}-\widetilde H_\varepsilon^{-1})\|_F
+
2\|M(X-\widetilde X_\varepsilon)\|_F
\\
&\quad+
2\|(M-\widetilde M)\widetilde X_\varepsilon\|_F.
\end{align*}
By \eqref{eq:regpolar:Hinv_bound} and
\eqref{eq:regpolar:X_bound}, the matrices
$\widetilde X_\varepsilon$ remain bounded for sufficiently small $\varepsilon$.
Hence,
\[
\|\overline M-\overline{\widetilde M}_\varepsilon\|_F
\leq
c\left(
\varepsilon^{1/2}+\varepsilon^{1/4}+\varepsilon
\right)
\leq
c\varepsilon^{1/4},
\]
which proves the claim.
\end{proof}

\section{Eigenvalue Decomposition Derivatives}
\label{appendix:eigenvalue_dev}

\begin{proposition}[Existing gradient of EVD, from~\cite{giles2008extended}]
\label{prop:ed_gradient}
Let $M \in \mathbb{R}^{n \times n}$ be a symmetric matrix with EVD
\[
M = V \Sigma V^\top,
\]
where $V$ is orthogonal and $\Sigma = \mathrm{diag}(\sigma_1,\dots,\sigma_n)$.
Given a scalar loss function $L = L(V,\Sigma)$, the gradient of $L$ with respect to $M$
is given by
\begin{equation*}
\frac{\partial L}{\partial M}
=
V
\left(
K^\top \odot \left( V^\top \frac{\partial L}{\partial V} \right)
+
\frac{\partial L}{\partial \Sigma}
\right)
V^\top.
\label{eq:ed_grad}
\end{equation*}

The matrix $K \in \mathbb{R}^{n \times n}$ is defined elementwise as
\begin{equation}
\label{eq:defK}
K_{ij}
=
\begin{cases}
\dfrac{1}{\sigma_i - \sigma_j}, & i \neq j, \\[6pt]
0, & i = j.
\end{cases}
\end{equation}
\end{proposition}

\begin{proof}[Proof of Theorem~\ref{thm:eigh_backward_blocks}]
We solve the minimization problem in
Definition~\ref{def:minimum-norm-evd-gradient}.
For the factor perturbations $dV=V\Omega$ and $d\Sigma=D$,
\[
    dM
    =dV\,\Sigma V^\top+V\,d\Sigma\,V^\top+V\Sigma\,dV^\top
    =V(\Omega\Sigma-\Sigma\Omega+D)V^\top.
\]
For a symmetric candidate $G$, set
\[
    H:=V^\top GV,
    \qquad
    B:=\frac12\bigl(V^\top\overline V-\overline V^\top V\bigr).
\]
Then $H^\top=H$ and $\|G\|_F=\|H\|_F$. Moreover,
\[
\begin{aligned}
    \langle G,dM\rangle_F
    &=\langle H\Sigma-\Sigma H,\Omega\rangle_F
      +\langle H,D\rangle_F,\\
    \langle\overline V,V\Omega\rangle_F
    &=\langle B,\Omega\rangle_F,
\end{aligned}
\]
where the second identity follows from $\Omega^\top=-\Omega$.
Thus, the constraints in the definition are equivalent to
\[
    \langle H\Sigma-\Sigma H-B,\Omega\rangle_F
    +\langle H-\overline\Sigma,D\rangle_F=0
\]
for every diagonal $D$ and every skew-symmetric $\Omega$
with $\Omega_{ij}=0$ whenever $\sigma_i=\sigma_j$.

Taking $\Omega=0$ and varying $D$ gives
\[
    H_{ii}=\overline\Sigma_{ii},
    \qquad 1\le i\le n.
\]
Taking $D=0$ and varying the permitted entries of $\Omega$ gives
\[
    (\sigma_j-\sigma_i)H_{ij}=B_{ij}
    \qquad\text{whenever }\sigma_i\ne\sigma_j.
\]
Consequently, the constraints are equivalent to
\[
    H_{ii}=\overline\Sigma_{ii},
    \qquad
    H_{ij}=\frac{B_{ij}}{\sigma_j-\sigma_i}
    \quad\text{if }\sigma_i\ne\sigma_j,
\]
with no restrictions on $H_{ij}$ when $i\ne j$ and
$\sigma_i=\sigma_j$. The prescribed entries are consistent
with symmetry because $B^\top=-B$.

Since
\[
    \|G\|_F^2=\|H\|_F^2
    =\sum_i H_{ii}^2+2\sum_{i<j}H_{ij}^2,
\]
the unique minimum-norm solution is obtained by setting
all unrestricted entries to zero. Hence
\[
    H_\star=K^\top\odot B+\overline\Sigma
    =S+\overline\Sigma,
\]
and therefore
\[
    \overline M_\star
    =VH_\star V^\top
    =V(S+\overline\Sigma)V^\top,
\]
as claimed.
\end{proof}

\section{Existing formulas for SVD Derivatives}
\label{appendix:existing}

\begin{theorem}[Backward differentiation of low-rank SVD]
\label{thm:lowrank_svd_backward}
Let $A \in \mathbb{R}^{m \times n}$ admit a rank-$k$ singular value decomposition
\[
A = U S V^\top,
\qquad
U \in \mathbb{R}^{m \times k},\;
S = \mathrm{diag}(s_1,\dots,s_k),\;
V \in \mathbb{R}^{n \times k},
\]
with $U^\top U = V^\top V = I_k$ and $s_i > 0$ pairwise distinct.

Let $\overline{U}, \overline{S}, \overline{V}$ denote gradients
corresponding to $U,S,V$, respectively.
Define the matrix $F \in \mathbb{R}^{k \times k}$ as
\begin{equation}
\label{eq:defF}
F_{ij}
=
\begin{cases}
\dfrac{1}{s_j^2 - s_i^2}, & i \neq j,\\[6pt]
0, & i=j.
\end{cases}
\end{equation}

Then the gradient of a scalar loss $\mathcal{L}$ with respect to $A$ is given by
\begin{align*}
\overline{A}
&=
\Bigl[
U \bigl( F \odot ( U^\top \overline{U} - \overline{U}^\top U ) \bigr) S
+
(I_m - UU^\top)\,\overline{U}\,S^{-1}
\Bigr] V^\top
\\[4pt]
&\quad
+
U \bigl( I_k \odot \overline{S} \bigr) V^\top
\\[4pt]
&\quad
+
U \Bigl[
S \bigl( F \odot ( V^\top \overline{V} - \overline{V}^\top V ) \bigr)V^\top
+
S^{-1}\,\overline{V}^\top (I_n - VV^\top)
\Bigr].
\end{align*}
\end{theorem}

\begin{proof}
    In work~\cite{townsend2016differentiating}.
\end{proof}

\begin{corollary}[Full-rank SVD backward]
\label{cor:fullrank_svd_backward} Let $m \ge n$.
If $k = n$, then the orthogonal complements vanish,
\[
I_n - VV^\top = 0,
\]
and the backward formula simplifies to
\begin{align}
\overline{A}
&=
U \bigl( F \odot ( U^\top \overline{U} - \overline{U}^\top U ) \bigr) S V^\top
- (I_m - UU^\top)\overline{U}S^{-1}V^\top
\\
&\quad
+
U \bigl( I_k \odot \overline{S} \bigr) V^\top
+
U S \bigl( F \odot ( V^\top \overline{V} - \overline{V}^\top V ) \bigr) V^\top.
\end{align}
\end{corollary}

\section{Backward differentiation of SVD in degenerate case}

\begin{theorem}[SVD backward formula in the degenerate case]
\label{theorem:svd_degenerate}
Let
\[
    M=U\Sigma V^\top,
    \qquad
    \Sigma=\operatorname{diag}(\sigma_1,\dots,\sigma_n),
\]
where equal singular values may have multiplicities. Define
\[
    A:=U^\top \overline U,
    \qquad
    B:=V^\top \overline V.
\]
Introduce the matrices \(F,G\in\mathbb R^{n\times n}\) by
\[
    F_{ij}
    =
    \begin{cases}
        \dfrac{1}{\sigma_j^2-\sigma_i^2}, & \sigma_i\neq \sigma_j,\\[1.2ex]
        0, & \sigma_i=\sigma_j,
    \end{cases}
\]
and
\[
    G_{ij}
    =
    \begin{cases}
        \dfrac{1}{2\sigma_i}, & \sigma_i=\sigma_j>0,\\[1.2ex]
        0, & \sigma_i\neq \sigma_j \ \text{or}\ \sigma_i=\sigma_j=0.
    \end{cases}
\]
Then
\[
    \overline M
    =
    U
    \left[
        F\odot
        \Big(
            (A-A^\top)\Sigma
            +
            \Sigma(B-B^\top)
        \Big)
        +
        \overline\Sigma
        +
        G\odot
        \operatorname{skew}(A-B)
    \right]
    V^\top,
\]
\end{theorem}

\begin{lemma}
Let \(M\in\mathbb R^{n\times n}\) have a full singular value decomposition
\[
    M=U\Sigma V^\top ,
\]
where \(U,V\in\mathbb R^{n\times n}\) are orthogonal. Assume that the singular
values are grouped into distinct blocks:
\[
    U=[U_1~\cdots~U_k],
    \qquad
    V=[V_1~\cdots~V_k],
\]
and
\[
    \Sigma
    =
    \operatorname{diag}
    \left(
        \sigma_1 I_{r_1},
        \dots,
        \sigma_k I_{r_k}
    \right),
    \qquad
    \sigma_i\neq \sigma_j \quad (i\neq j),
\]
Then the adjoint with respect to \(M\) has the form
\[
    \overline M=URV^\top ,
\]
where \(R\) is the block matrix with blocks \(R_{ij}\in\mathbb R^{r_i\times r_j}\)
defined as follows. For \(i\neq j\),
\[
    R_{ij}
    =
    \frac{
        \sigma_j
        \left(
            U_i^\top \overline U_j
            -
            \overline U_i^\top U_j
        \right)
        +
        \sigma_i
        \left(
            V_i^\top \overline V_j
            -
            \overline V_i^\top V_j
        \right)
    }{
        \sigma_j^2-\sigma_i^2
    },
    \qquad i\neq j.
\]

For the diagonal blocks with \(\sigma_i>0\),
\[
    R_{ii}
    =
    \overline\Sigma_{ii}
    +
    \frac{1}{2\sigma_i}
    \operatorname{skew}
    \left(
        U_i^\top \overline U_i
        -
        V_i^\top \overline V_i
    \right).
\]
For a block with \(\sigma_i=0\), we take
\[
    R_{ii}
    =
    \overline\Sigma_{ii}.
\]
\end{lemma}

\begin{proof}
Let
\[
    C:=U^\top \mathrm dM V,
    \qquad
    \Omega_U:=U^\top \mathrm dU,
    \qquad
    \Omega_V:=V^\top \mathrm dV.
\]
Since \(U\) and \(V\) are orthogonal, we have
\[
    \Omega_U^\top=-\Omega_U,
    \qquad
    \Omega_V^\top=-\Omega_V.
\]

Differentiating
\[
    M=U\Sigma V^\top
\]
gives
\[
    \mathrm dM
    =
    \mathrm dU\,\Sigma V^\top
    +
    U\,\mathrm d\Sigma\,V^\top
    +
    U\Sigma\,\mathrm dV^\top .
\]
Multiplying by \(U^\top\) from the left and by \(V\) from the right, we obtain
\[
    C
    =
    \Omega_U\Sigma+\mathrm d\Sigma-\Sigma\Omega_V.
\]
In block form this reads
\[
    C_{ij}
    =
    \sigma_j(\Omega_U)_{ij}
    -
    \sigma_i(\Omega_V)_{ij}
    +
    (\mathrm d\Sigma)_{ij}.
\]

For \(i\neq j\), the block \((\mathrm d\Sigma)_{ij}\) vanishes. Therefore,
\[
    C_{ij}
    =
    \sigma_j(\Omega_U)_{ij}
    -
    \sigma_i(\Omega_V)_{ij}.
\]
Using the skew-symmetry of \(\Omega_U\) and \(\Omega_V\), the transposed block equation is
\[
    C_{ji}^\top
    =
    -\sigma_i(\Omega_U)_{ij}
    +
    \sigma_j(\Omega_V)_{ij}.
\]
Thus, for each off-diagonal pair \(i\neq j\),
\[
    \begin{pmatrix}
        C_{ij}\\
        C_{ji}^\top
    \end{pmatrix}
    =
    \begin{pmatrix}
        \sigma_j I & -\sigma_i I\\
        -\sigma_i I & \sigma_j I
    \end{pmatrix}
    \begin{pmatrix}
        (\Omega_U)_{ij}\\
        (\Omega_V)_{ij}
    \end{pmatrix}.
\]
Since \(\sigma_i\neq \sigma_j\), this system is nonsingular, and hence
\[
    (\Omega_U)_{ij}
    =
    \frac{
        \sigma_j C_{ij}
        +
        \sigma_i C_{ji}^\top
    }{
        \sigma_j^2-\sigma_i^2
    },
\]
\[
    (\Omega_V)_{ij}
    =
    \frac{
        \sigma_i C_{ij}
        +
        \sigma_j C_{ji}^\top
    }{
        \sigma_j^2-\sigma_i^2
    }.
\]

Now compute the differential of the loss:
\[
    \mathrm d\mathcal L
    =
    \langle \overline U,\mathrm dU\rangle
    +
    \langle \overline\Sigma,\mathrm d\Sigma\rangle
    +
    \langle \overline V,\mathrm dV\rangle .
\]
Define
\[
    A:=U^\top \overline U,
    \qquad
    B:=V^\top \overline V.
\]
The contributions from the rotations inside the \(U\)- and \(V\)-coordinates are
\[
    \langle \overline U,\mathrm dU\rangle
    =
    \langle U^\top \overline U,\Omega_U\rangle
    =
    \langle A,\Omega_U\rangle,
\]
and
\[
    \langle \overline V,\mathrm dV\rangle
    =
    \langle V^\top \overline V,\Omega_V\rangle
    =
    \langle B,\Omega_V\rangle.
\]

Fix \(i\neq j\). Since \(\Omega_U\) is block skew-symmetric,
\[
    (\Omega_U)_{ji}=-(\Omega_U)_{ij}^\top .
\]
Therefore the two blocks \((i,j)\) and \((j,i)\) contribute
\[
    \langle A_{ij},(\Omega_U)_{ij}\rangle
    +
    \langle A_{ji},(\Omega_U)_{ji}\rangle
    =
    \left\langle
        A_{ij}-A_{ji}^\top,
        (\Omega_U)_{ij}
    \right\rangle .
\]
Similarly,
\[
    \langle B_{ij},(\Omega_V)_{ij}\rangle
    +
    \langle B_{ji},(\Omega_V)_{ji}\rangle
    =
    \left\langle
        B_{ij}-B_{ji}^\top,
        (\Omega_V)_{ij}
    \right\rangle .
\]

Substituting the formulas for \((\Omega_U)_{ij}\) and \((\Omega_V)_{ij}\), we get
\[
\begin{aligned}
&\left\langle
    A_{ij}-A_{ji}^\top,
    (\Omega_U)_{ij}
\right\rangle
+
\left\langle
    B_{ij}-B_{ji}^\top,
    (\Omega_V)_{ij}
\right\rangle
\\
&=
\left\langle
    A_{ij}-A_{ji}^\top,
    \frac{
        \sigma_j C_{ij}
        +
        \sigma_i C_{ji}^\top
    }{
        \sigma_j^2-\sigma_i^2
    }
\right\rangle
+
\left\langle
    B_{ij}-B_{ji}^\top,
    \frac{
        \sigma_i C_{ij}
        +
        \sigma_j C_{ji}^\top
    }{
        \sigma_j^2-\sigma_i^2
    }
\right\rangle .
\end{aligned}
\]
Collecting the coefficient of \(C_{ij}\), we obtain
\[
    R_{ij}
    =
    \frac{
        \sigma_j(A_{ij}-A_{ji}^\top)
        +
        \sigma_i(B_{ij}-B_{ji}^\top)
    }{
        \sigma_j^2-\sigma_i^2
    },
    \qquad i\neq j.
\]

It remains to identify the diagonal block. For the \(i\)-th diagonal block,
\[
    C_{ii}
    =
    \sigma_i\left((\Omega_U)_{ii}-(\Omega_V)_{ii}\right)
    +
    \mathrm d\sigma_i I_{r_i}.
\]
Set
\[
    Z_i:=(\Omega_U)_{ii}-(\Omega_V)_{ii}.
\]
Since both \((\Omega_U)_{ii}\) and \((\Omega_V)_{ii}\) are skew-symmetric,
\(Z_i\) is skew-symmetric. Hence
\[
    C_{ii}
    =
    (\mathrm d\Sigma)_{ii}
    +
    \sigma_i Z_i,
    \qquad
    Z_i^\top=-Z_i.
\]
Here \((\mathrm d\Sigma)_{ii}\) is diagonal, because it contains the
differentials of the singular values in the \(i\)-th block.

Therefore,
\[
    (\mathrm d\Sigma)_{ii}
    =
    \operatorname{diag}(\operatorname{diag}(C_{ii})),
\]
and, for \(\sigma_i>0\),
\[
    Z_i
    =
    \frac{1}{\sigma_i}\operatorname{skew}(C_{ii}).
\]

The diagonal contribution to the differential is
\[
\begin{aligned}
    &\langle A_{ii},(\Omega_U)_{ii}\rangle
    +
    \langle B_{ii},(\Omega_V)_{ii}\rangle
    +
    \langle \overline\Sigma_{ii},(\mathrm d\Sigma)_{ii}\rangle .
\end{aligned}
\]
Write
\[
    T_i:=\frac{1}{2}\left((\Omega_U)_{ii}+(\Omega_V)_{ii}\right).
\]
Then
\[
    (\Omega_U)_{ii}=T_i+\frac12 Z_i,
    \qquad
    (\Omega_V)_{ii}=T_i-\frac12 Z_i.
\]
Thus
\[
\begin{aligned}
    &\langle A_{ii},(\Omega_U)_{ii}\rangle
    +
    \langle B_{ii},(\Omega_V)_{ii}\rangle
    \\
    &=
    \langle A_{ii}+B_{ii},T_i\rangle
    +
    \frac12\langle A_{ii}-B_{ii},Z_i\rangle .
\end{aligned}
\]
The common rotation \(T_i\) is a gauge direction inside the degenerate singular
subspace. Since this direction does not change \(M\), the loss must be
invariant with respect to it, and hence its contribution vanishes:
\[
    \langle A_{ii}+B_{ii},T_i\rangle=0.
\]
Therefore the observable diagonal contribution is
\[
    \frac12\langle A_{ii}-B_{ii},Z_i\rangle
    +
    \langle \overline\Sigma_{ii},(\mathrm d\Sigma)_{ii}\rangle .
\]

Substituting
\[
    Z_i=\frac{1}{\sigma_i}\operatorname{skew}(C_{ii}),
    \qquad
    (\mathrm d\Sigma)_{ii}
    =
    \operatorname{diag}(\operatorname{diag}(C_{ii})),
\]
we obtain
\[
\begin{aligned}
    &\frac12\langle A_{ii}-B_{ii},Z_i\rangle
    +
    \langle \overline\Sigma_{ii},(\mathrm d\Sigma)_{ii}\rangle
    \\
    &=
    \left\langle
        \frac{1}{2\sigma_i}\operatorname{skew}(A_{ii}-B_{ii}),
        C_{ii}
    \right\rangle
    +
    \langle \overline\Sigma_{ii},C_{ii}\rangle .
\end{aligned}
\]
Hence, for \(\sigma_i>0\),
\[
    R_{ii}
    =
    \overline\Sigma_{ii}
    +
    \frac{1}{2\sigma_i}
    \operatorname{skew}(A_{ii}-B_{ii}).
\]

If \(\sigma_i=0\), then the relative rotation \(Z_i\) is not observable from
\(M\), because it is multiplied by \(\sigma_i\). Thus only the singular-value
part contributes, and we take
\[
    R_{ii}
    =
    \overline\Sigma_{ii}.
\]
Therefore,
\[
    \mathrm d\mathcal L
    =
    \langle URV^\top,\mathrm dM\rangle,
\]
which proves
\[
    \overline M=URV^\top .
\]
\end{proof}

\section{Complete SVD Forward and Backward Pipeline}
\label{appendix:pipeline}

Algorithm~\ref{alg:svd_pipeline} combines the CANS SVD forward pass
with the regularized EVD and polar backward formulas.
The backward pass is a custom rule evaluated from the returned
factors, with $W=UV^\top$, and reuses the eigenpairs computed
in the forward pass. At rank-deficient inputs, this specifies
a regularized backward convention.
We write $A_{\mathrm{sym}}=(A+A^\top)/2$ and take
$\overline{\Sigma}$ to be diagonal.
Since the forward pass returns $V^\top$, its incoming gradient
is transposed to obtain $\overline V$.

\begin{algorithm}[H]
\caption{CANS SVD with a regularized backward pass}
\label{alg:svd_pipeline}
\begin{algorithmic}[1]
\REQUIRE
$M\in\mathbb{R}^{m\times n}$ ($m\geq n$),
tolerance $\varepsilon_{\texttt{QR}}>0$,
regularization parameter $\varepsilon>0$
\ENSURE
SVD factors $U,\Sigma,V^\top$ and a regularized backward rule
\STATE \textbf{Forward pass}
\STATE Compute $U,\Sigma,V^\top$ using
Algorithm~\ref{alg:svd_forward} with tolerance
$\varepsilon_{\texttt{QR}}$
\STATE Cache $M,U,\Sigma,V$ and return $U,\Sigma,V^\top$
\STATE \textbf{Backward pass}
\STATE Receive incoming gradients
$\overline U,\overline\Sigma,\overline V$
\STATE $W\gets UV^\top$
\STATE $\overline W\gets\overline U V^\top$
\STATE $\overline V_{\mathrm{tot}}
\gets\overline V+W^\top\overline U$
\STATE $B\gets\tfrac12
\bigl(V^\top\overline V_{\mathrm{tot}}
-\overline V_{\mathrm{tot}}^\top V\bigr)$
\STATE $(K_\varepsilon)_{ij}\gets
\dfrac{\sigma_i-\sigma_j}
{(\sigma_i-\sigma_j)^2+\varepsilon^{1/2}}$,
$1\leq i,j\leq n$
\STATE $\overline H\gets
V\bigl(\overline\Sigma+K_\varepsilon^\top\odot B\bigr)V^\top$
\STATE $R_\varepsilon\gets
V\Sigma\bigl(\Sigma^2+\varepsilon^{1/2}I\bigr)^{-1}V^\top$
\STATE $C_\varepsilon\gets
\bigl(\overline H
-R_\varepsilon M^\top\overline W R_\varepsilon\bigr)_{\mathrm{sym}}$
\STATE $(T_\varepsilon)_{ij}\gets
\dfrac{\sigma_i+\sigma_j}
{(\sigma_i+\sigma_j)^2+\varepsilon^{1/4}}$,
$1\leq i,j\leq n$
\STATE $X_\varepsilon\gets
V\bigl(T_\varepsilon\odot(V^\top C_\varepsilon V)\bigr)V^\top$
\STATE $\overline M_\varepsilon\gets
\overline W R_\varepsilon+2MX_\varepsilon$
\RETURN $\overline M_\varepsilon$
\end{algorithmic}
\end{algorithm}

\end{document}